\documentclass[english]{article}
\usepackage[T1]{fontenc}
\usepackage[latin9]{inputenc}
\usepackage{geometry}
\usepackage{xcolor}
\usepackage{babel}
\usepackage{verbatim}
\usepackage{prettyref}
\usepackage{float}
\usepackage{mathtools}
\usepackage{algorithm2e}
\usepackage{amsmath}
\usepackage{amsthm}
\usepackage{amssymb}
\usepackage{graphicx}
\usepackage{subfig}
\usepackage[authoryear]{natbib}
\usepackage[unicode=true,pdfusetitle,
 bookmarks=true,bookmarksnumbered=false,bookmarksopen=false,
 breaklinks=true,pdfborder={0 0 1},backref=false,colorlinks=true]
 {hyperref}
\hypersetup{
 linkcolor=blue,citecolor=blue,urlcolor=blue,filecolor=blue,pdfpagelayout=OneColumn,pdfnewwindow=true,pdfstartview=XYZ,plainpages=false,pdfpagelabels,hyperindex=true}

\makeatletter

\floatstyle{ruled}
\newfloat{algorithm}{tbp}{loa}
\providecommand{\algorithmname}{Algorithm}
\floatname{algorithm}{\protect\algorithmname}

\theoremstyle{definition}
    \ifx\thechapter\undefined
      \newtheorem{defn}{\protect\definitionname}
    \else
      \newtheorem{defn}{\protect\definitionname}[chapter]
    \fi
\theoremstyle{plain}
    \ifx\thechapter\undefined
      \newtheorem{assumption}{\protect\assumptionname}
    \else
      \newtheorem{assumption}{\protect\assumptionname}[chapter]
    \fi
\theoremstyle{plain}
    \ifx\thechapter\undefined
      \newtheorem{prop}{\protect\propositionname}
    \else
      \newtheorem{prop}{\protect\propositionname}[chapter]
    \fi
\theoremstyle{plain}
    \ifx\thechapter\undefined
	    \newtheorem{thm}{\protect\theoremname}
	  \else
      \newtheorem{thm}{\protect\theoremname}[chapter]
    \fi
\theoremstyle{plain}
    \ifx\thechapter\undefined
  \newtheorem{cor}{\protect\corollaryname}
\else
      \newtheorem{cor}{\protect\corollaryname}[chapter]
    \fi
\theoremstyle{plain}
    \ifx\thechapter\undefined
  \newtheorem{prob}{\protect\problemname}
\else
      
    \fi
\theoremstyle{plain}
    \ifx\thechapter\undefined
  \newtheorem{remark}{\protect\remarkname}
\else
      \newtheorem{remark}{\protect\remarkname}[chapter]
    \fi

\@ifundefined{date}{}{\date{}}
\usepackage{babel}

\usepackage{stackengine}
\usepackage{bbm}
\usepackage{enumitem}

\usepackage{authblk}

\setlist{leftmargin=*, topsep=0.5em, parsep=0pt, itemsep=1em, labelindent=0pt, align=left}

\@ifundefined{showcaptionsetup}{}{%
 \PassOptionsToPackage{caption=false}{subfig}}
\usepackage{subfig}
\makeatother

\providecommand{\assumptionname}{Assumption}
\providecommand{\corollaryname}{Corollary}
\providecommand{\definitionname}{Definition}
\providecommand{\propositionname}{Proposition}
\providecommand{\remarkname}{Remark}
\providecommand{\theoremname}{Theorem}
\providecommand{\problemname}{Problem}

\providecommand{\keywords}[1]{\textit{Keywords:} #1}

\usepackage{nameref}

\newtheorem{problem}{Problem}

\begin{document}
\title{Portfolio optimization with a\\ prescribed terminal wealth distribution}

\author[1,2]{Ivan Guo}
\author[3]{Nicolas Langren\'e}
\author[1,2]{Gr\'egoire Loeper}
\author[1]{Wei Ning}

\affil[1]{\small School of Mathematical Sciences, Monash University, Melbourne, Australia}
\affil[2]{\small Centre for Quantitative Finance and Investment Strategies, Monash University, Australia}
\affil[3]{\small Data61, Commonwealth Scientific and Industrial Research Organisation, Australia}

\maketitle
\begin{abstract}
This paper studies a portfolio allocation problem, 
where the goal is to reach a prescribed wealth distribution at final time. We study this problem with the tools of optimal mass transport. We provide a dual formulation which we solve by a gradient descent algorithm. This involves solving an associated Hamilton-Jacobi-Bellman and Fokker--Planck equation by a finite difference method. 
Numerical examples for various prescribed terminal distributions are given, 
showing that we can successfully reach attainable targets.
We next consider adding consumption during the investment process, to take into account distributions that are either not attainable, or sub-optimal.
\end{abstract}

\section{Introduction}

Over the past decades, there has been a vast amount of research
on portfolio allocation. Perhaps the most iconic result is the portfolio selection
theory by \citet{markowitz1952portfolio}, which states that investors
should determine the allocation of wealth on the basis of the trade-off
between \textit{return} and \textit{risk}. The classical objective
function in a portfolio optimization problem is to maximize the expected return given variance level.
However, the first and second moments of the return of a portfolio is only a simplified description of the wealth. 
Researchers
then introduced objective functions that include more moments, such
as skewness, to provide a more accurate statistic description of the
distribution of the return (see, for example, \citealt{kraus1976skewness} and \citealt{lee1977functional}).

The whole distribution of the portfolio wealth would provide investors a
complete information, and instead
of optimizing the first moments of the distribution, our paper introduces
an objective function which includes a target distribution of the terminal
wealth. We address the problem of controlling the portfolio allocation process to reach the prescribed terminal distribution. Of course, as we will see not all distributions are attainable. 

On the one hand, this problem can be categorized as a stochastic control problem. The state variable is influenced by a process whose value is decided at any time $t \in [0,T]$, and we define such a process as a control. We can treat the portfolio allocation process as a control in the investment process. We aim to design the time path of the portfolio allocation process such that it steers the portfolio wealth from an initial state to a prescribed terminal distribution. 

One the other hand, designing  a continuous semimartingale having 
prescribed distributions at given times can be addressed with the optimal mass transport
(OMT) theory. The optimal transport problem is an old problem first addressed in the work
of \citet{monge1781memoire}, and was later revisited by \citet{kantorovich1942translocation} leading to the so-called Monge--Kantorovich formulation. 
A comprehensive review of the extensions and applications of the Monge--Kantorovitch
problem can be found in the book by \citet{rachev1998mass} and the books of \citet{villani2003topics, villani2008optimal}. The original formulation of the problem looks for a map $f:X\to Y$ that pushes a distribution $\mu$ to another distribution $\nu$.
Later, \citet{benamou2000computational}
reinterpreted the  problem in a fluid mechanics
framework, where one is not looking only for an optimal transport map, but instead for the whole trajectory of the mass distribution over time. This contribution opened the way to the problem of continuous optimal transport.

Stochastic extensions of the discrete and time continuous OMT problem have then flourished, see e.g. \citet{mikami2006duality}, \citet{tan2013optimal}, \citet{mikami2015two}, \citet{henry2016explicit}.
Beyond its mathematical interest, the optimal mass transport problem has applications in many fields, in economy,
meteorology, astrophysics (\citealt{brenier2003reconstruction}, \citealt{loeper2006reconstruction}), image processing (\citealt{ferradans2014regularized}), finance (\citealt{dolinsky2014martingale}, \citealt{henry2017model}). 

The novelty of this paper is to provide a new
perspective on portfolio optimization inspired by OMT. An investor must decide how to allocate her portfolio between a risky and a risk-free asset. The price of the risky asset is modelled by a semimartingale, with prescribed drift and diffusion coefficients. By controlling the portfolio allocation, she wants the distribution of the wealth to match, or be close to, a given target distribution. Depending on the risky asset diffusion coefficients, not all target distributions are attainable (think for example of too high an expected return versus variance), or optimal (one could reach a ``better'' distribution than the target). We consider two different approaches: either relaxing the terminal constraint by penalization, or adding a consumption process, whereby the investor can either inject or withdraw cash from the portfolio in order to reach the target. 

The rest of the paper is organized as follows. 
In \prettyref{sec:Problem-Formulation}, we formulate the problem. Then we introduce  the dual formulation in \prettyref{sec:Duality}.
In \prettyref{sec:Numerical-Methods}, we provide a gradient descent  algorithm
to solve the dual problem, and the numerical results are presented
in \prettyref{sec:Numerical-Results}. 
We give examples for general target distributions with various penalty functionals in \prettyref{sec:Penalty-function-with-intensity}. We consider the addition of consumption/cash input  in \prettyref{sec: consumption_part} and \prettyref{sec: cash_input_part}.

\section{\label{sec:Problem-Formulation}Problem Formulation}

Let $\mathcal{D}$ be a Polish space equipped with its Borel $\sigma$-algebra.
We denote $C(\mathcal{D};\mathbb{R})$ the space of continuous functions
on $\mathcal{D}$ with values in $\mathbb{R}$, $C_{b}(\mathcal{D};\mathbb{R})$
the space of bounded continuous functions and $C_{0}(\mathcal{D};\mathbb{R})$
the space of continuous functions, vanishing at infinity. Let $\mathcal{P}(\mathcal{D})$
be the space of Borel probability measures on $\mathcal{D}$ with
a finite second moment. Denote by $\mathcal{M}(\mathcal{D};\mathbb{R})$
the space of finite signed measures on $\mathcal{D}$ with values
in $\mathbb{R}$, $\mathcal{M}_{+}(\mathcal{D};\mathbb{R})\subset\mathcal{M}(\mathcal{D};\mathbb{R})$
be the subset of non-negative measures. When $\mathcal{D}$ is compact,
the topological dual of $C_{b}$ is given by $C_{b}(\mathcal{D};\mathbb{R})^{*}=\mathcal{M}(\mathcal{D};\mathbb{R})$.
But when $\mathcal{D}$ is non-compact, $C_{b}(\mathcal{D};\mathbb{R})^{*}$ is larger than $\mathcal{M}(\mathcal{D};\mathbb{R})$.
For convenience, we often use the notation $\mathcal{E}\coloneqq[0,1]\times\mathbb{R}$. 
We say that a function $\phi:\mathcal{E}\rightarrow\mathbb{R}$ belongs
to $C_{b}^{1,2}(\mathcal{E})$ if $\phi\in C_{b}(\mathcal{E})$ and
$(\partial_{t}\phi,\partial_{x}\phi,\partial_{xx}\phi)\in C_{0}(\mathcal{E};\mathbb{R},\mathbb{R},\mathbb{R})$.
Let $\mathbb{R}^{+}$ denote non-negative real numbers, and $\mathbb{S}^{d}$ denote the set of symmetric positive semidefinite matrices.

Let $\Omega\coloneqq(\omega\in C([0,1];\mathbb{R}^{d})),$
we denote by $\mathbb{F}=(\mathcal{F}_{t})_{t\in[0,1]}$ the filtration
generated by the canonical process. The
process $W$ is a $d$-dimensional standard Brownian motion on the filtered
probability space $\left(\Omega,\mathcal{F},\mathbb{F},\mathbb{P}\right)$.

We consider a portfolio with $d$ risky assets and one risk-free asset,
the risk-free interest $r$ being set to 0 for simplicity. We assume the drift $\mu:\mathcal{E} \rightarrow \mathbb{R}^{d}$ and covariance matrix $\Sigma: \mathcal{E} \rightarrow \mathbb{S}^{d}$ of the risky assets are known Markovian processes. Without loss of generality, we set
the time horizon $T$ to be $1$. The price process of the risky assets
is denoted by $S_{t}\in\mathbb{R}^{d}$ $(0\leq t\leq1)$, and the
$i$th element of $S_{t}$ follows the semimartingale 
\begin{equation}
\frac{dS_{t}^{i}}{S_{t}^{i}}=\mu_t^{i}dt+\sum_{j=1}^{d}\sigma_t^{ij}dW_{t}^{j},\quad1\leq i\leq d,\label{eq:dynamics of S}
\end{equation}
where $\sigma_t \coloneqq\Sigma_t^{\frac{1}{2}}\in\mathbb{R}^{d\times d}$
is the diffusion coefficient matrix.

The process $\alpha = (\alpha_t)_{t \in [0,1]}$ is a Markovian control. For $t\in[0,1]$, the portfolio allocation strategy $\alpha_{t}\in\mathbb{R}^{d}$
represents the proportion of the total wealth invested into the $d$
risky assets, and $1-\sum_{i=1}^{d}\alpha_{t}^{i}$ is the proportion
invested in the risk-free asset. We define the concept of \textit{admissible
control} as follows.
\begin{defn}
\label{def: An-admissible-alpha}An admissible control process $\alpha$
for the investor on $[0,1]$ is a progressively measurable process
with respect to $\mathbb{F}$, taking values in a compact convex set
$K\subset\mathbb{R}^{d}$. The set of all admissible $\alpha$ is
compact and convex, denoted by $\mathcal{\mathcal{K}}$.
\end{defn}
We denote by $X_{t}\in\mathbb{R}$ the portfolio wealth at time $t$. Starting from an initial wealth $x_{0}$, the wealth of the self-financing
portfolio evolves as follows, 
\begin{alignat}{1}
dX_{t} & =X_{t}\alpha_{t}^{\intercal}\mu_t dt+X_{t}\alpha_{t}^{\intercal}\sigma_t dW_{t},\label{eq:SDE}\\
X_{0} & =x_{0}.
\end{alignat}

\subsection{Portfolio optimization with a prescribed terminal distribution}

We denote by $\rho_{t}\coloneqq\mathbb{P}\circ X_{t}^{-1}\in\mathcal{P}(\mathbb{R})$
the distribution of $X_{t}$. In this problem, we know the initial
distribution of the portfolio wealth $\rho_{0}\in\mathcal{P}(\mathbb{R})$,
we are given a prescribed terminal distribution $\bar{\rho}_{1}\in\mathcal{P}(\mathbb{R})$
and a convex cost function $f(\alpha_{t}):K \rightarrow\mathbb{R}$.

With $\rho_{0}$ and a process $\alpha$, the realized
terminal distribution of the portfolio wealth is $\rho_{1}\coloneqq\mathbb{P}\circ X_{1}^{-1}$ ($\rho_{1}$
is not necessarily the same as $\bar{\rho}_{1}$). 
We want $\rho_{1}$ to be close to our target $\bar{\rho}_{1}$,
hence we introduce a functional $C(\rho_{1},\bar{\rho}_{1})$ to penalize
the deviation of $\rho_{1}$ from $\bar{\rho}_{1}$. 
At the same time, we want to minimize the expectation of the transportation cost from $\rho_{0}$ to $\rho_{1}$.
Combining the expected transportation cost and the penalty functional, our objective
function is
\begin{alignat}{1}
\inf_{\alpha,\rho}\left\{ \int_{\mathcal{E}}f(\alpha_{t})d\rho(t,x)+C(\rho_{1},\bar{\rho}_{1})\right\} ,\label{eq:primal-2}
\end{alignat}
where the feasible $(\alpha,\rho)$ in (\ref{eq:primal-2}) should
satisfy the initial distribution
\begin{alignat}{2}
\rho(0,x) & =\rho_{0}(x) & \quad\forall x\in\mathbb{R},\label{eq:initial marginal}
\end{alignat}
and the Fokker--Planck equation 
\begin{alignat}{2}
\partial_{t}\rho(t,x)+\partial_{x}(\alpha_{t}^{\intercal}\mu_t x\rho(t,x))-\frac{1}{2}\partial_{xx}(\alpha_{t}^{\intercal}\Sigma_t\alpha_{t}x^{2}\rho(t,x)) & =0 & \quad\forall(t,x)\in \mathcal{E}.\label{eq:Fokker-Planck}
\end{alignat}

However, the feasible set for $(\alpha,\rho)$ defined by equality
(\ref{eq:Fokker-Planck}) is not convex, which means we may not be
able to find the optimal solution. To address this issue, we introduce
the following definition.
\begin{defn}
\label{defn:def_A_B} We define maps $\tilde{B},\tilde{A}:\mathcal{E}\rightarrow\mathbb{R}$
as $\tilde{B}(t,x)\coloneqq\alpha_{t}^{\intercal}\mu_t x$ and $\tilde{A}(t,x)\coloneqq\alpha_{t}^{\intercal}\Sigma_t\alpha_{t}x^{2}$.
Then define $B(t,x)\coloneqq\tilde{B}\rho$, $B\in\mathcal{M}(\mathcal{E};\mathbb{R})$
and $A(t,x)\coloneqq\tilde{A}\rho$, $A\in\mathcal{M}^{+}(\mathcal{E};\mathbb{R})$.
Measures $B$ and $A$ are absolutely continuous with respect to $\rho$. 
\end{defn}

We show that $B$ and $A$ are connected in the following way:
\begin{prop}
\label{prop:The-necessary-and-sufficient}
When $d>1$ (resp. $d=1$), the necessary and sufficient
condition for the existence of an $\alpha_{t} \in \mathbb{R}^d$ satisfying Definition \prettyref{defn:def_A_B}  is $A\geq\frac{B^{2}}{\left\Vert \nu_t \right\Vert ^{2}\rho}$
(resp. $A=\frac{B^{2}}{\left\Vert \nu_t \right\Vert ^{2}\rho}$), where
$\nu_t \coloneqq\Sigma_t^{-\frac{1}{2}}\mu_t$.
\end{prop}
\begin{proof}
See \prettyref{sec:proof_AB_relation}.
\end{proof}

Using notations $\rho,B$ and $A$, the Fokker--Planck equation (\ref{eq:Fokker-Planck})
becomes linear and the SDE of the portfolio wealth reads
\begin{align}
dX_{t} & = \tilde{B}(t,X_t) dt + \tilde{A}^{\frac{1}{2}}(t,X_t) dW_{t},\label{eq:SDE_AB} \\
X_{0}  & = x_{0}.
\end{align}
From Proposition \prettyref{prop:The-necessary-and-sufficient}, at time $t$, it is possible that  the optimal drift $\tilde{B}(t,x)$ is not saturated, i.e., $\tilde{B}(t,x)^{ 2}<\left\Vert \nu_t \right\Vert ^{2}\tilde{A}(t,x)$.  It means that the drift function in SDE \eqref{eq:SDE_AB} can be greater and eventually we can reach a better terminal wealth with a higher expectation.
Then, instead of using this unsaturated drift $\tilde{B}(t,x)$ to reach the prescribed terminal distribution, we can use the drift $\tilde{B}(t,x) = \left\Vert \nu_t\right\Vert \sqrt{\tilde{A}} $ to attain a more ambitious distribution, and the extra part in the drift can be interpreted as cash saving.
In this case, even when we have multiple assets ($d>1$) in the portfolio, optimal portfolios should lie on the curve $\tilde{B}(t,x) = \left\Vert \nu_t\right\Vert \sqrt{\tilde{A}} $, as in the $d=1$ case. Any portfolio lying below the curve represents a less than ideal investment because for the same level of risk (variance), we could achieve a greater return. 
This is consistent with the efficient frontier in modern portfolio theory (\citealt{markowitz1952portfolio}).

\begin{comment}
If we add up the cash saving and the portfolio wealth $X_t$, we can get a new process \textit{wealth with saving} $X^c_t$, which follows the dynamics
\begin{alignat}{1}
dX^c_{t} & =\left\Vert \nu_t\right\Vert \sqrt{\tilde{A}}dt+\tilde{A}^{\frac{1}{2}} dW_{t},\\
X^c_{0} & =x_{0}.
\end{alignat}
\end{comment}

Now we define the concept of \textit{cash saving} at time $t$ as $c_t \coloneqq \left\Vert \nu_t\right\Vert \sqrt{\tilde{A}(t,x)} -\tilde{B}(t,x)$.
When the prescribed terminal distribution is not ambitious enough, to ensure we have as much cash saving as we can, we define the new feasible set as $\varPi\coloneqq\{(\rho,B,A):A\geq\frac{(B^+)^{2}}{\left\Vert \nu_t\right\Vert ^{2}\rho}\}$ ($B^+ \coloneqq \max(0, B)$) and we can see the set $\varPi$ is convex. 
To penalize measures out of the set $\varPi$, we define a
cost function $F:\mathcal{E}\times\mathbb{R}\times\mathbb{R}\rightarrow\mathbb{R}^{+}\cup\{+\infty\}$ such that
\begin{equation}
\label{eq:Def of F}
F\hspace{-0.1em}\left(\frac{B}{\rho},\frac{A}{\rho}\right) = f\hspace{-0.1em}\left(\frac{B}{\rho},\frac{A}{\rho}\right) + \delta(\rho, B, A),
\end{equation}
where $f:\mathcal{E}\times\mathbb{R}\times\mathbb{R}\rightarrow\mathbb{R}^{+}$  is a convex function and $\delta(\rho, B, A)$ is a delta function defined as
\begin{align}
\delta(\rho, B, A)=\begin{cases}
0 & \text{if } (\rho,B,A)\in \varPi,\\
+\infty & \text{otherwise}.
\end{cases}
\end{align}

Now we are ready to introduce formally the problem:
\begin{problem}
\label{prob:1}Starting from an initial distribution $\rho_{0}$, with a
prescribed terminal distribution $\bar{\rho}_{1}$ and a cost function (\ref{eq:Def of F}), we want to solve
the infimum of the functional
\begin{equation}
V(\rho_{0},\bar{\rho}_{1})=\inf_{\rho, B, A }\int_{\mathcal{E}}F\hspace{-0.1em}\left(\frac{B}{\rho},\frac{A}{\rho}\right)d\rho+C(\rho_{1},\bar{\rho}_{1})\label{eq:primal value function}
\end{equation}
over all  $(\rho,B,A)  \in \mathcal{M}( \mathcal{E}; \mathbb{R}\times\mathbb{R}\times\mathbb{R})$ satisfying the constraints
\begin{gather}
\partial_{t}\rho(t,x)+\partial_{x}B(t,x)-\frac{1}{2}\partial_{xx}A(t,x)=0 \quad\forall(t,x)\in \mathcal{E},\label{eq:FK equation}\\
\rho(0,x)=\rho_{0}(x) \quad\forall x\in\mathbb{R}.\label{eq:initial density}
\end{gather}
\end{problem}

\subsection{Assumptions}

We make the following assumptions which will hold throughout the paper.
\begin{assumption}
The probability measure $\rho_{t}$, $t\in(0,1]$ is absolutely continuous
with respect to the Lebesgue measure.
\end{assumption}
\begin{assumption}
\label{assu:C-is-continuous.} The penalty functional $C(\cdot,\bar{\rho}_{1}):\mathcal{P}(\mathbb{R})\rightarrow\mathbb{R}^{+}$
is lower semi-continuous and convex. We have $C(\rho_{1},\bar{\rho}_{1})=0$
if and only if $\rho_{1}(x)=\bar{\rho}_{1}(x)$ almost everywhere.
\end{assumption}
\begin{assumption}
$\,$ \label{assu:finiteness of F}
\begin{enumerate}[label=(\roman*)]
\item The function $F(\frac{B}{\rho},\frac{A}{\rho})$ is non-negative,
lower semi-continuous and strictly convex in $(\frac{B}{\rho},\frac{A}{\rho})$. 
\item The cost function $F(\frac{B}{\rho},\frac{A}{\rho})$ is coercive
in the sense that there exist constants $m>1$ and $K>0$ such that
\[
\left|\frac{B}{\rho}\right|^{m}+\left|\frac{A}{\rho}\right|^{m}\leq K\left(1+F\hspace{-0.1em}\left(\frac{B}{\rho},\frac{A}{\rho}\right)\right), \qquad \forall (t,x)\in \mathcal{E}.
\]
\item For all $(t,x)\in\mathcal{K}$, and for any $(\rho,B,A)\in\varPi$, we have 
\[
\int_{\mathcal{E}}\left|F\hspace{-0.1em}\left(\frac{B}{\rho},\frac{A}{\rho}\right)\right|d\rho<\infty,
\]
and 
\[
\mathbb{E}\left[\int_{0}^{1}\left|\frac{B}{\rho}\right|^{2}+\left|\frac{A}{\rho}\right|dt\right]<\infty.
\]
\end{enumerate}
\end{assumption}

For simplicity, we write $F(\frac{B}{\rho},\frac{A}{\rho}) \coloneqq F(t, x, \frac{B}{\rho},\frac{A}{\rho})$ if there is no ambiguity.

\section{\label{sec:Duality}Duality}

In this section, we introduce the dual problem to Problem \ref{prob:1},
this allows us to give optimality condition for the primal problem.
First of all, we find out the convex conjugate of the cost functional,
which will be used in the later proof.

\subsection{Convex Conjugate \label{subsec:Fenchel-Legendre-transformation}}

Define a function $G:C_{b}(\mathcal{E};\mathbb{R}\times\mathbb{R}\times\mathbb{R})\rightarrow\mathbb{R} \cup \{+\infty\}$
as 
\begin{alignat*}{1}
G(u,b,a) & =\sup_{\rho, \tilde{B}, \tilde{A}} \Bigl\{ u\rho+b\tilde{B}\rho + a\tilde{A}\rho -  F(\tilde{B},\tilde{A}) \rho \Bigr\}\\
 & =\sup_{\rho}\Bigl\{\rho\Bigl[u+\sup_{\tilde{A}\geq\frac{(\tilde{B}^+)^{2}}{\left\Vert \nu_t\right\Vert ^{2}}}\left(b\tilde{B}+a\tilde{A}-F(\tilde{B},\tilde{A})\right)\Bigr]\Bigr\}\\
 & =\sup_{\rho}\Bigl\{\rho\Bigl[u+F^{*}(b,a)\Bigr]\Bigr\},
\end{alignat*}
where $F^{*}$ is the convex conjugate of $F$. Since $\rho(t,x)$
is non-negative, it is obvious that 
\begin{align*}
G(u,b,a)=\begin{cases}
0 & \text{if }u+F^{*}(b,a)\leq0\quad\forall(t,x)\in\mathcal{K},\\
+\infty & \text{otherwise}.
\end{cases}
\end{align*}
If we restrict the domain of its convex conjugate $G^{*}:C_b^*(\mathcal{E};\mathbb{R}\times\mathbb{R}\times\mathbb{R}) \rightarrow  \mathbb{R} \cup \{ +\infty\}$ to $\mathcal{M} (\mathcal{E};\mathbb{R}\times\mathbb{R}\times\mathbb{R})$,
then
\begin{align}
G^{*}(\rho,B,A) & =\sup_{(u, b, a) \in C_{b}(\mathcal{E};\mathbb{R}\times\mathbb{R}\times\mathbb{R}) }\left\{ u\rho+bB+aA : u+F^{*}(b,a) \leq 0 \right\}. \label{eq:legendre transformation}
\end{align}
Because the function to be optimized is linear and $\rho(t,x) \geq 0$, we can see the optimal $u^{*}=-F^{*}(b,a)$ in (\ref{eq:legendre transformation}).
With $F$ being convex and  lower-semicontinuous, we have
\begin{align*}
G^{*}(\rho,B,A) & =\sup_{(b,a) \in C_{b}(\mathcal{E};\mathbb{R}\times\mathbb{R})}\left\{ -F^{*}(b,a)+b\tilde{B}+a\tilde{A} \right\} \rho\\
 & =F\hspace{-0.1em}\left(\frac{B}{\rho},\frac{A}{\rho}\right)\rho.
\end{align*}
The supremum is pointwise in time and space, and we can write 
\begin{equation}
\int_{\mathcal{E}}F\hspace{-0.1em}\left(\frac{B}{\rho},\frac{A}{\rho}\right)d\rho=\sup_{(u,b,a) \in C_{b}(\mathcal{E};\mathbb{R}\times\mathbb{R}\times\mathbb{R})}  \Bigl\{ \int_{\mathcal{E}}ud\rho +bdB+adA:u+F^{*}(b,a)\leq0\Bigr\}.\label{eq:equivalent integrant}
\end{equation}

\subsection{Dual Problem}

Now we can state our main result. A key element in the dual problem
is  the Hamilton--Jacobi--Bellman (HJB) equation: 
\begin{equation}
\partial_{t}\phi+\sup_{\tilde{A}\geq\frac{(\tilde{B}^+)^{2}}{\left\Vert \nu_t\right\Vert ^{2}}}\Bigl\{\partial_{x}\phi\tilde{B}+\frac{1}{2}\partial_{xx}\phi\tilde{A}-F(\tilde{B},\tilde{A})\Bigr\}=0.\label{eq:HJB_first}
\end{equation}
For any $\phi(t,x) \in C_{b}^{1,2}(\mathcal{E})$  solution of the HJB equation (\ref{eq:HJB_first}),
 It\^o's formula yields, 
\begin{align*}
\int_{\mathbb{R}}\phi_{1}d\rho_{1}-\phi_{0}d\rho_{0} & =\int_{\mathcal{E}}\left(\partial_{t}\phi+\partial_{x}\phi\tilde{B}+\frac{1}{2}\partial_{xx}\phi\tilde{A}\right)d\rho\\
 & =\int_{\mathcal{E}}\left(-F^{*}(\partial_{x}\phi,\frac{1}{2}\partial_{xx}\phi)+\partial_{x}\phi\tilde{B}+\frac{1}{2}\partial_{xx}\phi\tilde{A}\right)d\rho\\
 & \leq\int_{\mathcal{E}}F(\tilde{B},\tilde{A})d\rho.
\end{align*}
Adding the penalty functional to both sides yields 
\begin{align}
\int_{\mathbb{R}}\phi_{1}d\rho_{1}-\int_{\mathbb{R}}\phi_{0}d\rho_{0}+C(\rho_{1},\bar{\rho}_{1}) & \leq\int_{\mathcal{E}}F(\tilde{B},\tilde{A})d\rho+C(\rho_{1},\bar{\rho}_{1}).\label{eq:ito inequality}
\end{align}
Taking the infimum of the left hand side of (\ref{eq:ito inequality})
over $\rho_{1}$ and taking the infimum of the right hand side of (\ref{eq:ito inequality})
over $(\rho,B,A)$, we get 
\begin{align*}
-C^{*}(-\phi_{1})-\int_{\mathbb{R}}\phi_{0}d\rho_{0} 
& \leq\inf_{(\rho,B,A) \in \mathcal{M}( \mathcal{E}; \mathbb{R}\times\mathbb{R}\times\mathbb{R}) }\int_{\mathcal{E}}F\hspace{-0.1em}\left(\frac{B}{\rho},\frac{A}{\rho}\right)d\rho+C(\rho_{1},\bar{\rho}_{1}) \\
& \leq V(\rho_0, \bar{\rho}_1).
\end{align*}
The following result shows that optimizing the left hand side yields
an equality.
\begin{thm}
[Duality]\label{thm:Duality_theorem}

When $C(\rho_{1},\bar{\rho}_{1})$ is continuous, there holds
\begin{alignat}{1}
V(\rho_{0},\bar{\rho}_{1}) & =\sup_{\phi}\left\{ -C^{*}(-\phi_{1})-\int_{\mathbb{R}}\phi_{0}d\rho_{0}\right\} ,\label{eq:dual form}
\end{alignat}
where the supremum is taken over all $\phi(t,x)\in C_{b}^{1,2}(\mathcal{E})$
satisfying 
\begin{align}
\partial_{t}\phi(t,x)+\sup_{\tilde{A}\geq\frac{(\tilde{B}^+)^{2}}{\left\Vert \nu_t\right\Vert ^{2}}}\left\{ \partial_{x}\phi\tilde{B}+\frac{1}{2}\partial_{xx}\phi\tilde{A}-F(\tilde{B},\tilde{A})\right\}  & \le0, & \forall(t,x)\in[0,1]\times\mathbb{R}.\label{eq:constraint of dual problem}
\end{align}
\end{thm}
\begin{proof}
This proof is an application of the Fenchel--Rockafellar duality theorem,
e.g., \citet[Theorem 1.12]{brezis2010functional}.
From the constraint (\ref{eq:FK equation}), we have that for all $\phi \in C_{b}^{1,2}(\mathcal{E})$,
\begin{equation}
\int_{\mathbb{R}}\int_{0}^{1}\phi\partial_{t}\rho+\phi\partial_{x}B-\frac{1}{2}\phi\partial_{xx}Adtdx=0.
\label{eq:int FK}
\end{equation}
Integrating by parts we obtain
\begin{equation}
\int_{\mathbb{R}}\phi_{1}d\rho_{1}-\phi_{0}d\rho_{0}-\int_{\mathcal{E}}\partial_{t}\phi d\rho+\partial_{x}\phi dB+\frac{1}{2}\partial_{xx}\phi dA=0.\label{eq:Lagrangian penalty 1-1}
\end{equation}
Because of equation (\ref{eq:equivalent integrant}), we can reformulate
the primal problem (\ref{eq:primal value function}) as a saddle
point problem: 
\begin{alignat}{1}
V(\rho_{0},\bar{\rho}_{1}) & =\inf_{\rho,B,A}\sup_{u+F^{*}(b,a)\leq0}\int_{\mathcal{E}}ud\rho+bdB+adA+C(\rho_{1},\bar{\rho}_{1}).\label{eq:objective function_2}
\end{alignat}
Adding the Lagrangian penalty (\ref{eq:Lagrangian penalty 1-1}) to
the functional (\ref{eq:objective function_2}), then Problem \ref{prob:1}
can be written as
\begin{gather*}
V(\rho_{0},\bar{\rho}_{1})=\inf_{\rho,B,A}\sup_{u+F^{*}(b,a)\leq0,\phi}\int_{\mathcal{E}}ud\rho+bdB+adA+C(\rho_{1},\bar{\rho}_{1})\\
+\int_{\mathbb{R}}\phi_{1}d\rho_{1}-\phi_{0}d\rho_{0}-\int_{\mathcal{E}}\partial_{t}\phi d\rho+\partial_{x}\phi dB+\frac{1}{2}\partial_{xx}\phi dA.
\end{gather*}
We write $C^{*}(r): C_b(\mathbb{R}; \mathbb{R}) \rightarrow \mathbb{R} \cup \{+\infty\}$ for the convex conjugate of functional $C(\rho_{1},\bar{\rho}_{1})$:
\[
C^{*}(r)=\sup_{\rho_{1}\geq0}\Bigl\{\int_{\mathbb{R}}rd\rho_{1}-C(\rho_{1},\bar{\rho}_{1})\Bigr\}.
\]
Here we define the functional $\alpha:C_{b}(\mathcal{E};\mathbb{R}\times\mathbb{R}\times\mathbb{R}\times\mathbb{R})\rightarrow\mathbb{R}\cup\{+\infty\}$
by 
\begin{align}
\alpha(u,b,a,r)=\begin{cases}
C^{*}(r) & \text{ if }u+F^{*}(b,a)\leq0,\\
+\infty & \text{ otherwise }.
\end{cases}
\label{eq:alpha}
\end{align}
Its convex conjugate $\alpha^{*}:C_{b}^{*}(\mathcal{E};\mathbb{R}\times\mathbb{R}\times\mathbb{R}\times\mathbb{R})\rightarrow\mathbb{R}\cup\{+\infty\}$
is defined as
\begin{alignat}{1}
\alpha^{*}(\rho,B,A,\rho_{1}) & =\sup_{u+F^{*}(b,a)\leq0,r}\int_{\mathcal{E}}ud\rho+bdB+adA+\Bigl[\int_{\mathbb{R}}rd\rho_{1}-C^{*}(r)\Bigr].
\label{eq: alpha_star_def}
\end{alignat}
If we restrict the domain to $\mathcal{M} (\mathcal{E};\mathbb{R}\times\mathbb{R}\times\mathbb{R}\times\mathbb{R})$,
with \eqref{eq:equivalent integrant} and Assumption \ref{assu:C-is-continuous.}, we have
\begin{alignat*}{1}
\alpha^{*}(\rho,B,A,\rho_{1}) & =
\begin{cases}
\int_{\mathcal{E}}F(\frac{B}{\rho},\frac{A}{\rho})d\rho+C(\rho_{1},\bar{\rho}_{1}) & \text{if }\rho\in\mathcal{M}_{+}\text{ and }B=\tilde{B}\rho,A=\tilde{A}\rho,\\
+\infty  & \text{ otherwise}.
\end{cases}
\end{alignat*}
Indeed, if $\rho$ is not positive in \eqref{eq: alpha_star_def}, we would let $b = a = 0$ and $u = -\lambda \mathbbm{1}_O$ for some $O$ such that $\rho(O) <0$ and let $\lambda \rightarrow +\infty$. If $B$ or $A$ are not absolutely continuous with respect to $\rho$, we can find some $O$ such that $\rho(O) = 0$ but $B(O) \neq 0$ or $A(O) \neq 0$. Then we let $u = -F^*(b,a)$ and $b = a = \lambda \mathbbm{1}_O$,  and $\alpha^*(\rho, B, A) \geq  \lambda  B(O) +  \lambda  A(O) \rightarrow +\infty$ by letting $\lambda \rightarrow \pm \infty $ depending on the sign of $B(O)$ and $A(O)$. 

Next, we say that the set $(u,b,a,r)\in C_{b}(\mathcal{E};\mathbb{R}\times\mathbb{R}\times\mathbb{R}\times\mathbb{R})$
is \textit{represented} by $\phi\in C_{b}^{1,2}(\mathcal{E})$ if
\[
u=-\partial_{t}\phi,\quad b=-\partial_{x}\phi,\quad a=-\frac{1}{2}\partial_{xx}\phi,\quad r=\phi_{1}.
\]
Then define $\beta:C_{b}(\mathcal{E};\mathbb{R}\times\mathbb{R}\times\mathbb{R}\times\mathbb{R})\rightarrow\mathbb{R}\cup\{+\infty\}$
as follows, 
\begin{align}
\beta(u,b,a,r)=\begin{cases}
\int_{\mathbb{R}}\phi_{0}d\rho_{0} & \text{ if }(u,b,a,r)\text{ is represented by }\phi\in C_{b}^{1,2}(\mathcal{E}),\\
+\infty\qquad & \text{ otherwise}.
\end{cases}
\label{eq:beta}
\end{align}
Notice that $\beta$ is well-defined, indeed, it does not depend on
the choice of $\phi$. If both $\phi,\psi$ represent $u,b,a,r$,
then $\phi_{1}=\psi_{1}\forall x\in\mathbb{R}$, $\partial_{t}\phi(t,x)=\partial_{t}\psi(t,x)$,
$\partial_{x}\phi(t,x)=\partial_{x}\psi(t,x)$, $\partial_{xx}\phi(t,x)=\partial_{xx}\psi(t,x)$
$\forall(t,x)\in\mathcal{K}$. It follows that $\phi_{0}(x)=\psi_{0}(x)$
$\forall x\in\mathbb{R}$. The set of represented functions $(u,b,a,r)$
is a linear subspace, and $\beta$ is linear with respect to $(u,b,a,r)$
in the convex set. Hence $\beta$ is convex and its convex conjugate
$\beta^{*}:C_{b}^{*}(\mathcal{E};\mathbb{R}\times\mathbb{R}\times\mathbb{R}\times\mathbb{R})\rightarrow\mathbb{R}\cup\{+\infty\}$
is 
\begin{align*}
\beta^{*}(\rho,B,A,\rho_{1})= & \sup_{u,b,a,r}  \int_{\mathcal{E}}ud\rho+bdB+adA+\int_{\mathbb{R}} rd\rho_{1}-\phi_{0}d\rho_{0},\\
 & \quad \text{over all }(u,b,a,r)\in C_{b}(\mathcal{E};\mathbb{R}\times\mathbb{R}\times\mathbb{R}\times\mathbb{R}) \text{ represented by }\phi\in C_{b}^{1,2}(\mathcal{E}).
\end{align*}
Or equivalently, 
\begin{alignat*}{1}
\beta^{*}(\rho,B,A,\rho_{1}) & =\sup_{\phi}\int_{\mathcal{E}}-\partial_{t}\phi d\rho-\partial_{x}\phi dB-\frac{1}{2}\partial_{xx}\phi dA+\int_{\mathbb{R}}\phi_{1}d\rho_{1}-\phi_{0}d\rho_{0}.
\end{alignat*}
We find that $\beta^{*}(\rho,B,A,\rho_{1})=0$ if $(\rho,B,A,\rho_{1})$
satisfies (\ref{eq:Lagrangian penalty 1-1}), and $\beta^{*}(\rho,B,A,\rho_{1})=+\infty$
otherwise. 

Now we can express our objective functional $V(\rho_0, \bar{\rho}_1)$ as
\begin{alignat*}{1}
V(\rho_{0},\bar{\rho}_{1}) 
& ={\inf_{(\rho,B,A)\in \mathcal{M}(\mathcal{E};\mathbb{R}\times\mathbb{R}\times\mathbb{R})}\left\{ \alpha^{*}(\rho,B,A,\rho_{1})+\beta^{*}(\rho,B,A,\rho_{1})\right\} }\\
 & ={\inf_{(\rho,B,A,\rho_{1})\in \mathcal{M}(\mathcal{E};\mathbb{R}\times\mathbb{R}\times\mathbb{R}\times\mathbb{R})}\left\{ \alpha^{*}(\rho,B,A,\rho_{1})+\beta^{*}(\rho,B,A,\rho_{1})\right\} }\\
 & = \inf_{(\rho,B,A,\rho_{1})\in C_{b}^{*}(\mathcal{E};\mathbb{R}\times\mathbb{R}\times\mathbb{R}\times\mathbb{R})}\left\{ \alpha^{*}(\rho,B,A,\rho_{1})+\beta^{*}(\rho,B,A,\rho_{1})\right\}. 
\end{alignat*}
The second equality is because $\beta^*(\rho,B,A,\rho_{1}) = +\infty$ if $\rho_1$ does not equal to $\rho(t,x)$ at time $t=1$. We prove the third equality in \prettyref{sec:extend_non_compact}.

We can let $\phi(t,x)=t$, then $u=-1,b=0,a=0,r=1$. We can see $\alpha(-1,0,0,1)=1$ and it
is continuous in $(u,b,a,r)$ at this point, and $\beta(-1,0,0,1)=0$
being finite at this point.
Finally, the conditions of Fenchel duality theorem in \citet[Theorem 1.12]{brezis2010functional}
are fulfilled, and it implies
\begin{alignat*}{1}
V(\rho_{0},\bar{\rho}_{1}) 
 & = \inf_{(\rho,B,A,\rho_{1})\in C_{b}^{*}(\mathcal{E};\mathbb{R}\times\mathbb{R}\times\mathbb{R}\times\mathbb{R})}\left\{ \alpha^{*}(\rho,B,A,\rho_{1})+\beta^{*}(\rho,B,A,\rho_{1})\right\} \\
 & =\sup_{(u,b,a,r)\in C_{b}(\mathcal{E};\mathbb{R}\times\mathbb{R}\times\mathbb{R}\times\mathbb{R})}\left\{ -\alpha(-u,-b,-a,-r)-\beta(u,b,a,r)\right\},
 \end{alignat*}
over the set $(u,b,a,r)$ being represented by $\phi\in C_{b}^{1,2}(\mathcal{E})$,
and satisfying $-u+F^{*}(-b,-a)\leq0$. 

Therefore we express $V(\rho_{0},\bar{\rho}_{1}) $ in terms of $\phi$: 
\begin{alignat*}{1}
V(\rho_{0},\bar{\rho}_{1}) & =\sup_{(u,b,a,r)\in C_{b}(\mathcal{E};\mathbb{R}\times\mathbb{R}\times\mathbb{R})}\left\{ -C^{*}(-r)-\int_{\mathbb{R}}\phi_{0}d\rho_{0}\right\} \\
 & =\sup_{\phi \in C_{b}^{1,2}(\mathcal{E})  }\left\{ -C^{*}(-\phi_{1})-\int_{\mathbb{R}}\phi_{0}d\rho_{0}\right\} ,
\end{alignat*}
under the constraint $\partial_{t}\phi+F^{*}(\partial_{x}\phi,\frac{1}{2}\partial_{xx}\phi)\leq0.$\textcolor{brown}{{}
}\textcolor{black}{As a consequence of Fenchel duality theorem, the
infimum in the primal problem is attained if finite.} This completes
the proof.
\end{proof}
Actually, using the same proof as in \citet{guo2019calibration}, we can write the dual formulation in the following way:
\begin{cor}
\label{cor:duality corollary}When $C(\rho_{1},\bar{\rho}_{1})$ is continuous, there holds 
\begin{equation}
V(\rho_{0},\bar{\rho}_{1})=\sup_{\phi_{1}}\left\{ -C^{*}(-\phi_{1})-\int_{\mathbb{R}}\phi_{0}d\rho_{0}\right\} ,
\label{eq:dual with phi_1}
\end{equation}
where the supremum is running over all functions $\phi_{1}\in C_{b}^{2}(\mathbb{R})$,
and $\phi_{0}$ is a viscosity solution
of the Hamilton--Jacobi--Bellman equation
\begin{gather}
\begin{cases}
-\phi_{t}-\sup_{\tilde{A}\geq\frac{(\tilde{B}^+)^{2}}{\left\Vert \nu_t\right\Vert ^{2}}}\Bigl[\phi_{x}\tilde{B}+\frac{1}{2}\phi_{xx}\tilde{A}-F(\tilde{B},\tilde{A})\Bigr]=0, & \text{in }[0,1)\times\mathbb{R},\\
\phi(1,x)=\phi_{1}(x), & \text{on }[1]\times\mathbb{R}.
\end{cases}\label{eq:HJB PDE}
\end{gather}
\end{cor}

Because the minimal objective function \eqref{eq:primal value function}
is a trade-off between the cost function and the penalty functional,
the optimal $\phi_1$ in the dual problem (\ref{eq:dual with phi_1})
will not in general ensure that $\rho_{1}$ reaches $\bar{\rho}_{1}$, unless the penalty
functional goes to infinity for $\rho_1 \neq \bar{\rho}_1$.
When $\bar{\rho}_{1}$ is attainable, it can be realized by choosing the penalty
functional as an indicator function
\begin{equation}
C(\rho_{1},\bar{\rho}_{1})=\begin{cases}
0 & \text{ if }\rho_{1}=\bar{\rho}_{1},\\
+\infty & \;\text{if}\:\rho_{1}\neq\bar{\rho}_{1}.
\end{cases}\label{eq:characteristic function}
\end{equation}
Using the penalty functional \eqref{eq:characteristic function} is equivalent to adding the terminal constraint
$\rho_{1}=\bar{\rho}_{1},\, \forall x \in \mathbb{R}$.  This also recovers our problem to the classical
optimal transport problem. 
\begin{cor}
When $C(\rho_{1},\bar{\rho}_{1})$ is defined as (\ref{eq:characteristic function}),
there holds
\begin{alignat}{1}
V(\rho_{0},\bar{\rho}_{1}) & =\sup_{\phi_{1}}\left\{ \int_{\mathbb{R}}\phi_{1}d\bar{\rho}_{1}-\phi_{0}d\rho_{0}\right\} ,\label{eq:normal OT dual-1-1}
\end{alignat}
where the supremum is running over all $\phi_{1}\in C_{b}^{2}(\mathbb{R})$
and $\phi_{0}$ is a viscosity solution
of the Hamilton--Jacobi--Bellman equation (\ref{eq:HJB PDE}).
\end{cor}
\begin{proof}
This proof is very similar to the one of Theorem \ref{thm:Duality_theorem}, hence the repetitive steps are omitted here. 
Being different from the proof of Theorem \ref{thm:Duality_theorem}, in this case, we define the functional $\alpha:C_{b}(\mathcal{E};\mathbb{R}\times\mathbb{R}\times\mathbb{R}\times\mathbb{R})\rightarrow\mathbb{R}\cup\{+\infty\}$
by
\begin{alignat*}{1}
\alpha(u,b,a,r) & =\begin{cases}
\int_{\mathbb{R}}rd\bar{\rho}_{1} & \text{ if }u+F^{*}(b,a)\leq0,\\
+\infty & \text{ otherwise}.
\end{cases}
\end{alignat*}
Then its convex conjugate of $\alpha^*: C^*_{b}(\mathcal{E};\mathbb{R}\times\mathbb{R}\times\mathbb{R}\times\mathbb{R})\rightarrow\mathbb{R}\cup\{+\infty\}$ is 
\begin{alignat*}{1}
\alpha^{*}(\rho,B,A,\rho_{1}) & =\sup_{u+F^{*}(b,a)\leq0,r}\int_{\mathcal{E}}ud\rho+bdB+adA+ \int_{\mathbb{R}} rd\rho_{1}-rd\bar{\rho}_{1}.
\end{alignat*}
We restrict the domain to $\mathcal{M} (\mathcal{E};\mathbb{R}\times\mathbb{R}\times\mathbb{R}\times\mathbb{R})$, we have
\begin{alignat*}{1}
\alpha^{*}(\rho,B,A,\rho_{1}) 
 & =\int_{\mathcal{E}}F\hspace{-0.1em}\left(\frac{B}{\rho},\frac{A}{\rho}\right)d\rho+\sup_{r}\left\{ \int_{\mathbb{R}}r(d\rho_{1}-d\bar{\rho}_{1})\right\} \\
 & =\int_{\mathcal{E}}F\hspace{-0.1em}\left(\frac{B}{\rho},\frac{A}{\rho}\right)d\rho+C(\rho_{1},\bar{\rho}_{1}).
\end{alignat*}
Note that $\sup_{r}\int_{\mathbb{R}}r(d\rho_{1}-d\bar{\rho}_{1})$ is equal to $0$
if $\rho_{1}=\bar{\rho}_{1}\forall x\in\mathbb{R}$ and is equal to $+\infty$
otherwise, which is equivalent to $C(\rho_{1},\bar{\rho}_{1})$ in
(\ref{eq:characteristic function}). Define $\beta:C_{b}(\mathcal{E};\mathbb{R}\times\mathbb{R}\times\mathbb{R}\times\mathbb{R})\rightarrow\mathbb{R}\cup\{+\infty\}$ by \eqref{eq:beta}, 
let $\phi(t,x)=t$, the conditions of Fenchel duality theorem in \citet[Theorem 1.12]{brezis2010functional}
are fulfilled. Therefore, we get
\begin{alignat*}{1}
V(\rho_{0},\bar{\rho}_{1}) & =\sup_{(u,b,a,r) \in C_{b}(\mathcal{E};\mathbb{R}\times\mathbb{R}\times\mathbb{R}\times\mathbb{R})}\left\{ -\alpha(-u,-b,-a,-r) - \beta(u, b, a, r) \right\} \\
 & =\sup_{(u,b,a,r) \in C_{b}(\mathcal{E};\mathbb{R}\times\mathbb{R}\times\mathbb{R}\times\mathbb{R})}\left\{ \int_{\mathbb{R}}rd\bar{\rho}_{1}-\int_{\mathbb{R}}\phi_{0}d\rho_{0}\right\},
\end{alignat*}
over the set $(u,b,a,r)$ being represented
by $\phi\in C_{b}^{1,2}(\mathcal{E})$, and satisfying $-u+F^{*}(-b,-a)\leq0$.
For the same reasons as in Corollary  \ref{cor:duality corollary}, we can express $V(\rho_{0},\bar{\rho}_{1})$ in terms of $\phi$:
\begin{align*}
V(\rho_{0},\bar{\rho}_{1}) & =\sup_{\phi_{1}}\left\{ \int_{\mathbb{R}}\phi_{1}d\bar{\rho}_{1}-\int_{\mathbb{R}}\phi_{0}d\rho_{0}\right\} ,
\end{align*}
where $\phi_{0}(x)$ is \textcolor{black}{a viscosity solution }of
the Hamilton--Jacobi--Bellman equation (\ref{eq:HJB PDE}).
\end{proof}

\section{\label{sec:Numerical-Methods}Numerical Methods for the Dual Problem}

There has been a vast amount of numerical algorithms for the optimal
mass transport problem. Gradient descent based methods are widely
used to solve the reformulated dual problem of the Monge--Kantorovich
problem, for example, by \citet{chartrand2009gradient} and \citet{tan2013optimal}.
\citet{cuturi2013sinkhorn} looked at transport problems from a maximum
entropy perspective and computed the OT distance through Sinkhorn's
matrix scaling algorithm. This algorithm is also used for the entropic
regularization of optimal transport by \citet{benamou2019generalized}.

In this paper, we also use a gradient descent based method to solve
the dual problem in \prettyref{sec:Duality}. We know $\phi(t,x)$ is the
solution of the HJB equation (\ref{eq:HJB PDE}).
For a given terminal function $\phi_{1}$, we can calculate $\phi_{0}$
by solving the HJB equation backward.

\subsection{\label{subsec:forward_backward}Finite Difference Scheme}

First of all, to get $\phi(0,x)$, we solve the following PDE 
\begin{equation}
\partial_{t}\phi+\sup_{\tilde{A}\geq\frac{(\tilde{B}^+)^{2}}{\left\Vert \nu_t\right\Vert ^{2}}}\left\{ \partial_{x}\phi\tilde{B}+\frac{1}{2}\partial_{xx}\phi\tilde{A}-F(\tilde{A},\tilde{B})\right\} =0,\label{eq:HJB_example-2}
\end{equation}
with a given terminal boundary condition $\phi_{1}(x)$ backwardly, using an implicit
finite difference scheme. We let 
\begin{equation}
(\tilde{A}^{*},\tilde{B}^{*})=\underset{\tilde{A}\geq\frac{(\tilde{B}^+)^{2}}{\left\Vert \nu_t\right\Vert ^{2}}}{\mathrm{argmax}}\left\{ \partial_{x}\phi\tilde{B}+\frac{1}{2}\partial_{xx}\phi\tilde{A}-F(\tilde{A},\tilde{B})\right\} .
\label{eq: optimal_A_B}
\end{equation}
In the numerical setting, we use $N$ time steps and $M$
space grid points. We use a constant time step $\Delta t$ and
a constant spatial step $\Delta x$. We discretize the PDE \eqref{eq:HJB_example-2}
using a forward approximation for $\partial_t\phi$, a central approximation
for $\partial_x\phi$, and a standard approximation for $\partial_{xx}\phi$. With
some manipulation, we get the discretized form of \eqref{eq:HJB_example-2} as
\begin{alignat}{1}
\Bigl(\frac{\Delta t\tilde{A}_{i}^{*^n}}{2(\Delta x)^{2}}-\frac{\Delta t\tilde{B}_{i}^{*^n}}{2\Delta x}\Bigr)\phi_{i-1}^{n}+\Bigl(-1-\frac{\Delta t\tilde{A}_i^{*^n}}{(\Delta x)^{2}}\Bigr)\phi_{i}^{n}+\Bigl(\frac{\Delta t\tilde{A}_{i}^{*^n}}{2(\Delta x)^{2}}+\frac{\Delta t\tilde{B}_{i}^{*^n}}{2\Delta x}\Bigr)\phi_{i+1}^{n} & =-\phi_{i}^{n+1}+\Delta tF(\tilde{A}_{i}^{*^n},\tilde{B}_{i}^{*^n}),\label{eq:implicit form}
\end{alignat}
where the optimal controls $\tilde{A}^{*^n}$ and $\tilde{B}^{*^n}$ depend
on $\phi^{n}$. 
It is difficult to check the stability condition in our PDE 
because the optimal $\tilde{A}^{*^n},\tilde{B}^{*^n}$ are unknown, but
fortunately implicit finite difference methods have a weaker requirement
than explicit finite difference methods. 
At the $n$-th time step of the implicit finite difference method, although we do not have the true values for $\phi^{n}$, 
we can make an initial guess of $(\tilde{A}^{*^n},\tilde{B}^{*^n})_0$ using the known values $\phi^{n+1}$,
then use a fixed-point iteration scheme to generate a sequence $(\tilde{A}^{*^n},\tilde{B}^{*^n})_{k, k=1,2,...}$ until 
$(\tilde{A}^{*^n},\tilde{B}^{*^n})_{k}$ converges.
This method is also implemented in \citet{guo2019robust}.

With the optimal drift $\tilde{B}^{*}$ and diffusion $\tilde{A}^{*}$
known, we can now propagate forward with the Fokker--Planck equation \eqref{eq:FK equation}
to find the empirical terminal density $\rho_{1}$. 
With an initial wealth $x_{0}$, the initial distribution $\rho_{0}$ is a
Dirac Delta distribution $\delta(x-x_0)$. Since we used implicit finite difference
to solve the HJB equation \eqref{eq:HJB_example-2} backward, we use
an explicit scheme for the forward Fokker--Planck equation \eqref{eq:FK equation}.
Then the discretized form is
\begin{alignat}{1}
\frac{\rho_{i}^{n+1}-\rho_{i}^{n}}{\Delta t}+\frac{\tilde{B}_{i+1}^{*^{n}}\rho_{i+1}^{n}-\tilde{B}_{i-1}^{*^{n}}\rho_{i-1}^{n}}{2\Delta x}-\frac{1}{2}\frac{\tilde{A}_{i+1}^{*^{n}}\rho_{i+1}^{n}+\tilde{A}_{i-1}^{*^{n}}\rho_{i-1}^{n}-2\tilde{A}_{i}^{*^{n}}\rho_{i}^{n}}{\Delta x^{2}} & =0.\label{eq:FK equation discretized}
\end{alignat}

\subsection{Optimization algorithm}

A key role in the gradient descent method is the optimality condition. By providing a gradient, the computation is faster and more accurate. 
For convenience, we define another
function 
\begin{alignat}{1}
\tilde{V}(\phi_{1}) & \coloneqq C^{*}(-\phi_{1}) + \int_{\mathbb{R }}\phi_{0} d\rho_{0},\label{eq:value_minimize}
\end{alignat}
and $V(\rho_{0},\bar{\rho}_{1})=-\inf_{\phi_{1}}\tilde{V}(\phi_{1})$.
Then we need to find an optimal $\phi_{1}$ to minimize $\tilde{V}(\phi_{1})$.
The change of $\tilde{V}(\phi_{1})$ w.r.t $\phi_{1}$ is 
\begin{alignat}{1}
\delta\tilde{V}(\phi_{1}) 
& = \delta C^{*}(-\phi_{1}) + \int_{\mathbb{R}} \rho_{0}\frac{\delta\phi_{0}}{\delta\phi_{1}}\delta\phi_{1}dx,\\
& =\delta C^{*}(-\phi_{1}) + \int_{\mathbb{R}} \rho_{0}\delta\phi_{0}dx.
\label{eq:variation}
\end{alignat}
We know that $\phi(t,x)$ in \eqref{eq:value_minimize} satisfies $F^{*}\hspace{-0.1em}\left(\partial_{x}\phi,\frac{1}{2}\partial_{xx}\phi\right)=-\partial_{t}\phi.$
If we add a small variation $\delta\phi$ to $\phi$ and denote $\partial_{x}\phi$ as $p$ and $\frac{1}{2}\partial_{xx}\phi$ as $q$
for short,  
then we get $\partial_p F^{*}(p,q) \partial_{x}\delta\phi+ \frac{1}{2} \partial_q F^{*}(p,q) \partial_{xx}\delta\phi=-\partial_{t}\delta\phi,$
which is equivalent to 
\begin{equation}
\partial_{t}\delta\phi+\partial x\delta\phi\tilde{B}^{*}+\frac{1}{2}\partial_{xx}\delta\phi\tilde{A}^{*}=0.\label{eq:PDE of variation}
\end{equation}
Multiplying PDE (\ref{eq:PDE of variation}) by an arbitrary density
function $\rho(t,x)$ and with integration by parts, we have 
\[
\int_{\mathbb{R}}\rho_{1}\delta\phi_{1}-\rho_{0}\delta\phi_{0}dx-\int_{\mathbb{R}}\int_{0}^{1}\delta\phi\partial_{t}\rho+\delta\phi\partial_{x}(\rho\tilde{B}^{*})-\frac{1}{2}\delta\phi\partial_{xx}(\rho\tilde{A}^{*})dxdt=0.
\]
Since the equation $\partial_{t}\rho+\partial_{x}(\rho\tilde{B}) - \frac{1}{2}\partial_{xx}(\rho\tilde{A})=0$
holds for all admissible $(\tilde{A},\tilde{B})$, we get 
\begin{equation}
\int_{\mathbb{R}}\rho_{0}\delta\phi_{0}dx=\int_{\mathbb{R}}\rho_{1}\delta\phi_{1}dx.
\label{eq:replace phi_0 with phi_1}
\end{equation}
Substituting \eqref{eq:replace phi_0 with phi_1} into \eqref{eq:variation}, we can see an optimal terminal function $\phi_{1}$ should satisfy the optimality condition 
\begin{equation}
\nabla\tilde{V}(\phi_{1})=\frac{\delta C^{*}(-\phi_{1})}{\delta \phi_{1}}+\rho_{1}=0, \quad \forall x\in \mathbb{R}.\label{eq:optimality condition}
\end{equation}
\begin{remark}
When $C(\bar{\rho}_{1},\rho_{1})$ is defined as (\ref{eq:characteristic function}),
the corresponding optimality condition is 
\begin{equation}
\nabla\tilde{V}(\phi_{1})=-\bar{\rho}_{1}+\rho_{1}=0,  \quad \forall x\in \mathbb{R}.\label{eq:optimality condition_inf}
\end{equation}
\end{remark}
Now we are ready to solve the dual problem numerically. 
In Algorithm \ref{alg:Gradient-Descent}, we state the
gradient descent based algorithm to look for the optimal $\phi_{1}$ in \eqref{eq:dual with phi_1}.
It includes solving the HJB equation and the Fokker--Planck equation with a finite difference method
combined with a fixed-point iteration, as described in Section \prettyref{subsec:forward_backward}.
A similar numerical scheme can be found in \citet{guo2019calibration} for calibrating volatilities by optimal transport. 
\begin{algorithm}
Initial guess $\phi_{N}^{1}\coloneqq0$

\While{$1 \leq k \leq \text{max iteration}$ and $\bigl\Vert \nabla\tilde{V}(\phi_{N}^{k}) \bigr\Vert_{\infty} > \text{tolerance }$}{

Let $\phi_{N}=\phi_{N}^{k}$;

\For{time step $n=N-1 : 0$:}{

Let $\phi_n = \phi_{n+1}$, solve the PDE (\ref{eq:implicit form}) with $ ( \tilde{A}^{*^n}, \tilde{B}^{*^n} )_0$ obtained from 
\eqref{eq: optimal_A_B}.

Get the value vector $\phi_{n}^{0}$;

\While{$1 \leq j \leq \text{max iteration}$ and $\bigl\Vert\phi_{n}^{j}-\phi_{n}^{j-1}\bigr\Vert_{2}>\text{tolerance }$}{
Let $\phi_n = \phi_{n}^{j-1}$, solve the PDE (\ref{eq:implicit form}) with $( \tilde{A}^{*^n}, \tilde{B}^{*^n} )_j$ 
obtained from \eqref{eq: optimal_A_B}.

Get the value vector $\phi_{n}^{j}$;

$j=j+1$; }

Let $\phi_{n}=\phi_{n}^{j}$, store the optimal controls $( \tilde{A}^{*^n} , \tilde{B}^{*^{n}} ) = ( \tilde{A}^{*^n} , \tilde{B}^{*^{n}} )_j$;
} Compute the empirical distribution $\rho_{1}^{k}$ from
$\rho_{0}$ with Fokker--Planck equation (\ref{eq:FK equation discretized});

Compute the gradient vector $\nabla\tilde{V}(\phi_{N}^{k}) = \left( \frac{\delta C^{*}(-\phi_{N}^{k})}{\delta \phi_{N}^{k}} +\rho_{1}^{k}\right)\Delta x$;

Update $\phi_{N}^{k+1}$ with Quasi-Newton Method using the gradient
information $\nabla\tilde{V}(\phi_{N}^{k})$;

$k=k+1$; } The optimal $\phi_{N}=\phi_{N}^{k}$. 
\caption{A gradient descent based optimization scheme\label{alg:Gradient-Descent}}
\end{algorithm}

\section{\label{sec:Numerical-Results}Numerical Results}
In this section, we will apply Algorithm \ref{alg:Gradient-Descent} and demonstrate various numerical examples. We also consider the situations with cash saving and cash input during the investment process.

\subsection{Penalty functional  with an intensity parameter \label{sec:Penalty-function-with-intensity}}
Before we demonstrate the numerical results, we need to choose an
appropriate penalty functional $C(\rho_{1},\bar{\rho}_{1})$.
There is a range of methods to measure distribution discrepancy. 
A comprehensive survey on the distance or similarity measures between 
probability density functions (PDFs) is provided by \citet{cha2007comprehensive}. 
Note that our choice of penalty functional is not restricted to metrics, 
as long as $C(\rho_{1},\bar{\rho}_{1})$ satisfies
Assumption \ref{assu:C-is-continuous.} and describes similarity
of the two PDFs.

The most intuitive choice is the $L^2$ norm of the difference. This quadratic function
is convex and easy to implement. 
In the first example, we use the squared Euclidean distance as the penalty functional and $F(\tilde{A},\tilde{B})=(\tilde{A}-0.2)^2 + (\tilde{B}-0.2)^2$
as the cost function. We define the penalty functional as 
\[
C(\rho_{1},\bar{\rho}_{1})=\frac{\lambda}{2} \int_{\mathbb{R}} (\rho_{1}-\bar{\rho}_{1})^{2}dx,
\]
where the parameter $\lambda$ can be regarded as the intensity of
the penalty for the inconsistency.
Then the dual problem \eqref{eq:dual with phi_1} can be expressed explicitly
as 
\begin{alignat}{1}
V(\rho_{0},\bar{\rho}_{1}) & =\sup_{\phi_{1}}\left\{ \int_{\mathbb{R}}-\frac{1}{2\lambda}\phi_{1}^{2}+\phi_{1}\bar{\rho}_{1}-\rho_{0}\phi_{0}dx\right\} .\label{eq:intensity value}
\end{alignat}

\begin{comment}
By contradiction, we can prove that when $F$ is a function in $\tilde{A}$
only, the optimal solution always happens on the boundary, i.e., $\tilde{B}^{*2}=\left\Vert \nu_t\right\Vert ^{2}\tilde{A}^{*}$.
Hence we can rewrite the PDE (\ref{eq:HJB_example-2}) as 
\begin{equation}
\phi_{t}+\sup_{\tilde{B}}\left\{ \phi_{x}\tilde{B}+\frac{1}{2}\phi_{xx}\frac{\tilde{B}^{2}}{\left\Vert \nu_t\right\Vert ^{2}}-\frac{\tilde{B}^{4}}{\left\Vert \nu_t\right\Vert ^{4}}+0.2\frac{\tilde{B}^{2}}{\left\Vert \nu_t\right\Vert ^{2}} - 0.01 \right\} =0.\label{eq:continuous PDE}
\end{equation}

Then we need to find the optimal $\tilde{B}$ in PDE (\ref{eq:continuous PDE})
by solving for the real cubic root of the First Order Condition: 
\[
\phi_{x}+\left(\phi_{xx}+0.4\right)\frac{\tilde{B}^{*}}{\left\Vert \nu_t\right\Vert ^{2}}-4\frac{\tilde{B}^{*3}}{\left\Vert \nu_t\right\Vert ^{4}}=0\text{.}
\]
\end{comment}
In this and the following numerical examples, we set the initial wealth $x_{0}=5, \mu = 0.1, \sigma = 0.1$. 
Figures \ref{fig:(6,1)-lambda1} and \ref{fig:(6,1)-lambda20} 
compare the empirical distribution of
the terminal wealth ($\rho_{1}$) and the prescribed terminal
distribution ($\bar{\rho}_{1}$) for different intensities $\lambda$, where $\bar{\rho}_{1}=\mathcal{N}(6,1)$\footnote{We denote $\mathcal{N}(\mu, \sigma)$ a Normal distribution with mean $\mu$ and standard deviation $\sigma$.}.
We can see that $\rho_{1}$ gets closer to $\bar{\rho}_{1}$
as we increase the intensity of the penalty. 
In figures \ref{fig:(6,1)-lambda_inf} and \ref{fig:optimal-functions}, we use the penalty functional 
\eqref{eq:characteristic function}, which is equivalent to setting $\lambda = +\infty$. As shown in Figure \ref{fig:(6,1)-lambda_inf}, 
this penalty functional makes $\rho_{1}$ attain the target $\bar{\rho}_{1}$, 
and the plot \ref{fig:optimal-functions}
illustrates the optimal function $\phi_{1}$ and the corresponding $\phi_{0}$ we got from Algorithm \ref{alg:Gradient-Descent}. 
\begin{figure}
\subfloat[$\lambda=1$\label{fig:(6,1)-lambda1}]{\includegraphics[scale=0.7]{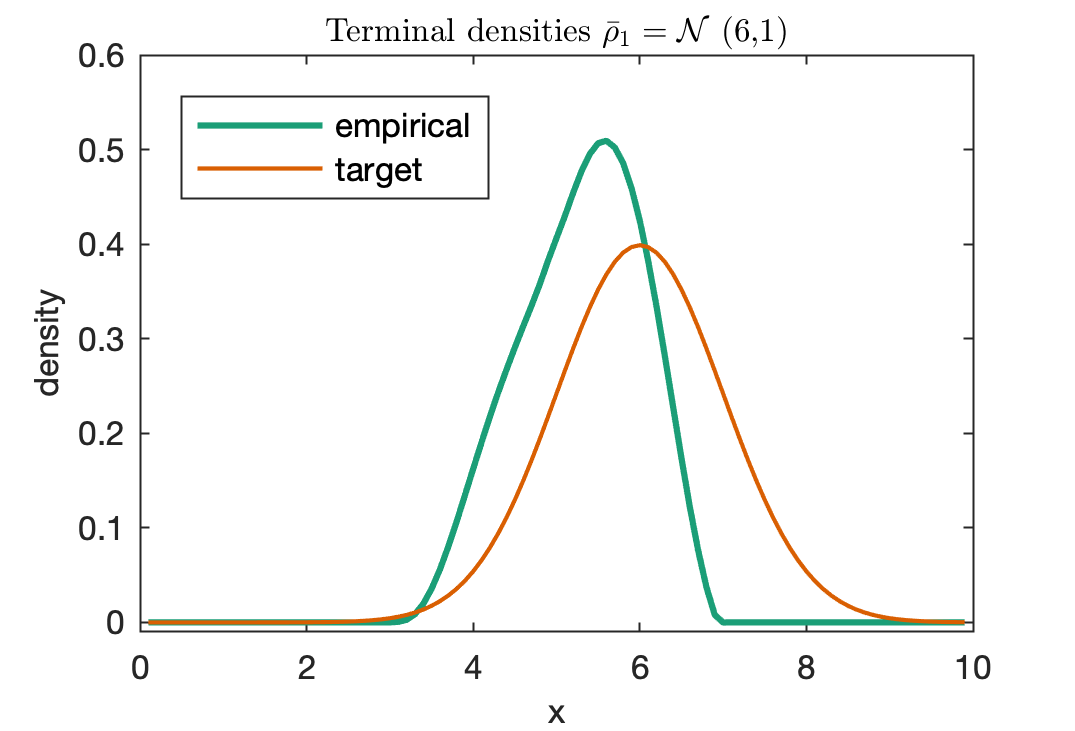}}
\qquad{}
\subfloat[$\lambda=20$\label{fig:(6,1)-lambda20}]{\includegraphics[scale=0.7]{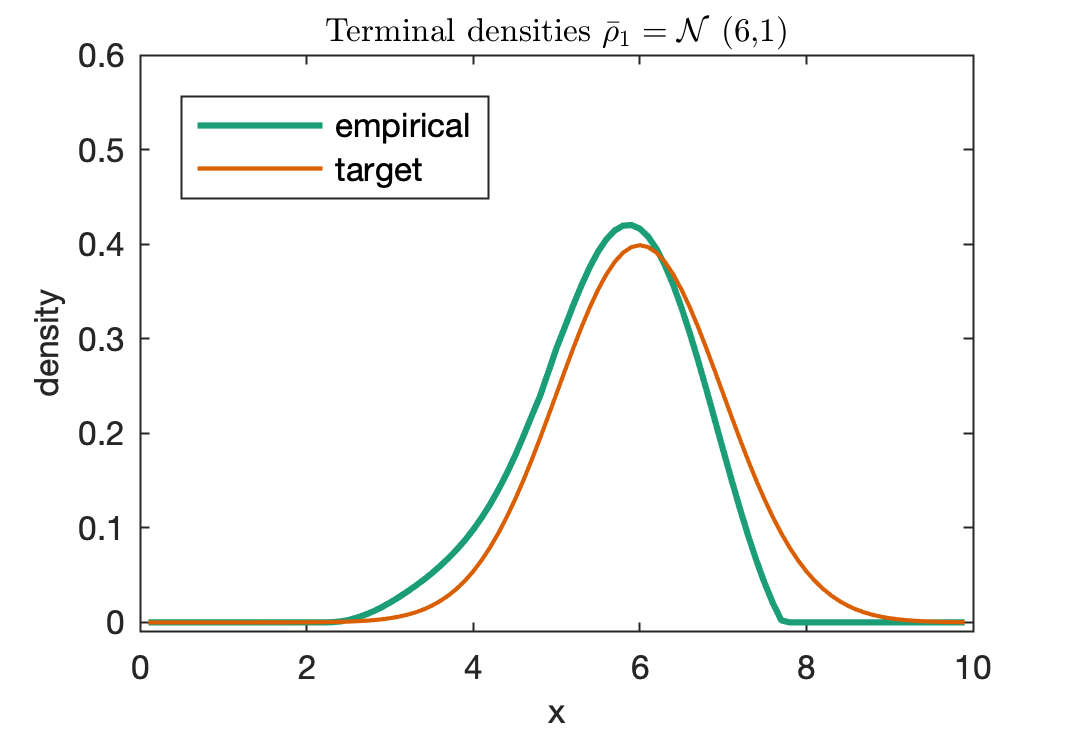}}

\subfloat[infinite penalty\label{fig:(6,1)-lambda_inf}]{\includegraphics[scale=0.7]{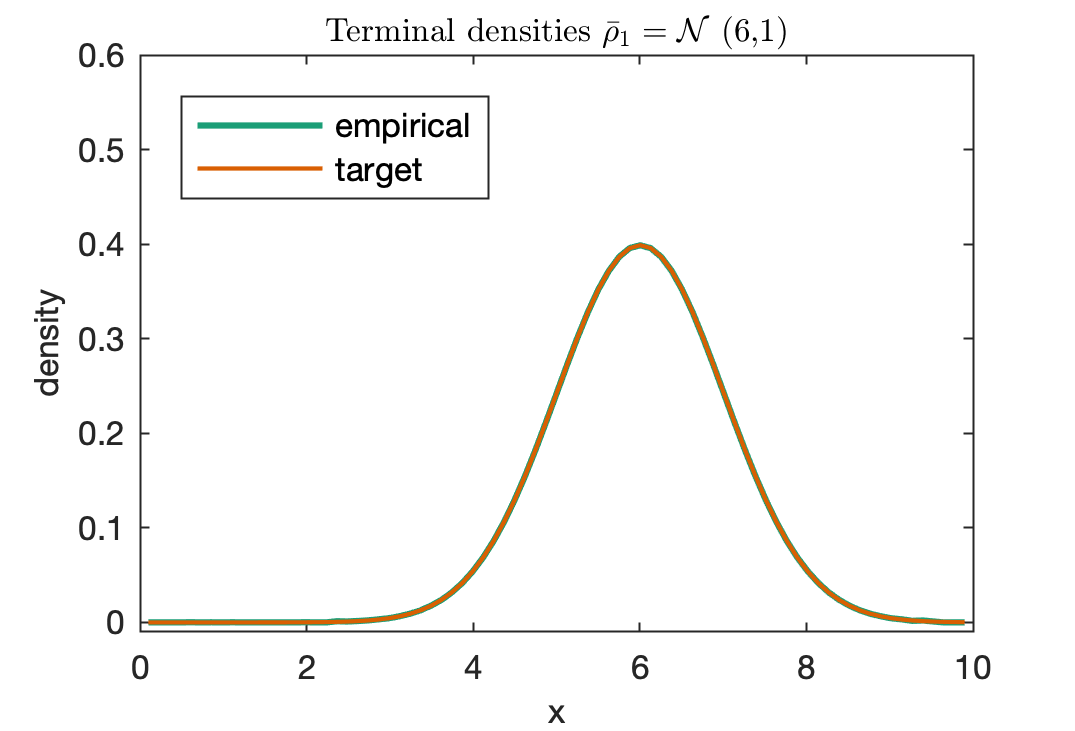}}
\qquad{}
\subfloat[optimal $\phi_{1}$ and $\phi_{0}$\label{fig:optimal-functions}]{\includegraphics[scale=0.7]{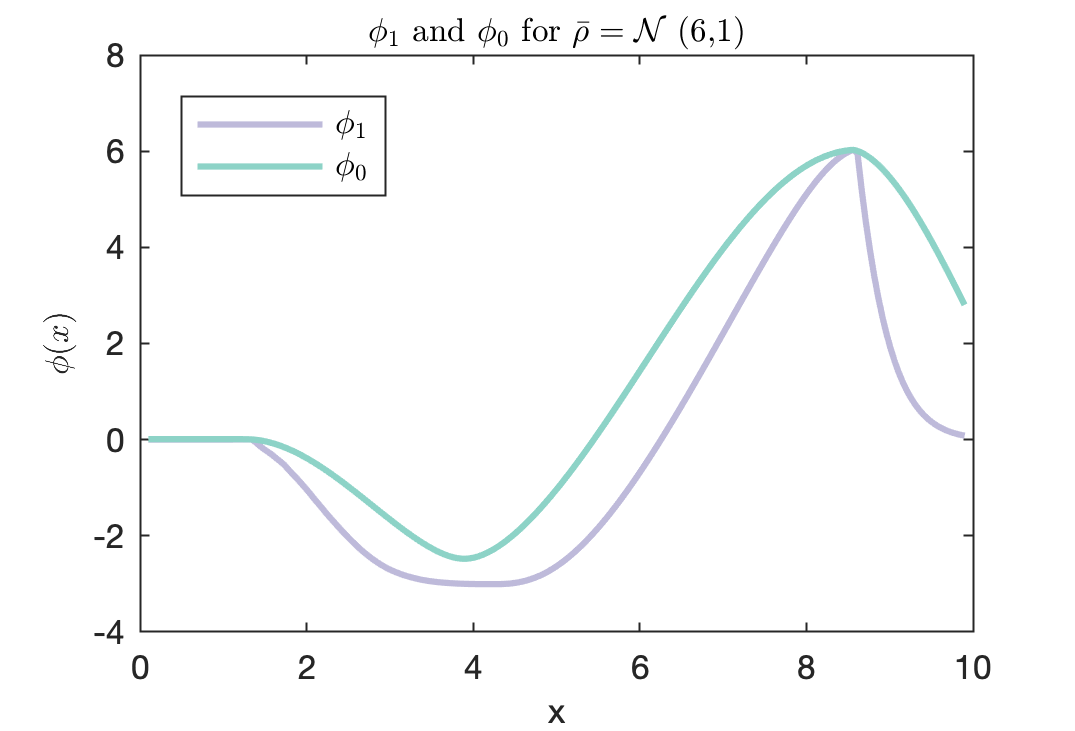}}
\caption{Attainable example: $\bar{\rho}_{1}=\mathcal{N}(6,1)$\label{fig:(6,1)}}
\end{figure}

Compared to other research where the prescribed distributions are
restricted to Gaussian, our method applies to a wide choice of $\bar{\rho}_{1}$,
such as heavy-tailed and asymmetric distributions. 
In Figure \ref{fig:mix_normal},
we illustrate an example where $\bar{\rho}_{1}$ is a mixture of
two Normal distributions, where 
\[
\bar{\rho}_{1}(x)=0.5\mathcal{N}(4,1)+0.5\mathcal{N}(7,1).
\]
In Figure \ref{fig:Distance_lambda}, we plot how the Euclidean distance
$\left( \int_{\mathbb{R}} (\rho_{1}-\bar{\rho}_{1})^{2}dx\right)^{\frac{1}{2}}$
changes with respect to $\lambda$ . As we increase the intensity
parameter $\lambda$, the Euclidean distance between $\rho_{1}$
and $\bar{\rho}_{1}$ decreases. As $\lambda$ goes to infinity, the
distance asymptotically goes to zero.
\begin{figure}[h]
\centering
\begin{minipage}{.5\textwidth}
  \centering
  \includegraphics[scale=0.7]{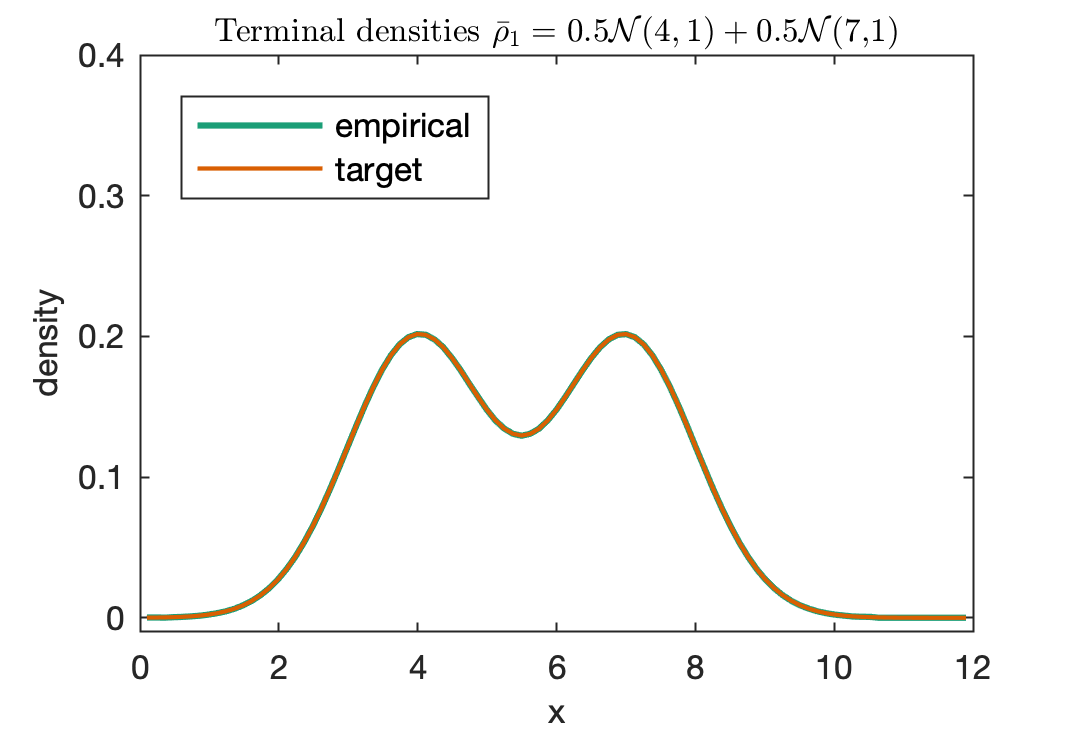} 
  \captionof{figure}{mixture of Normal distributions}
  \label{fig:mix_normal}
\end{minipage}%
\begin{minipage}{.5\textwidth}
  \centering
  \includegraphics[scale=0.7]{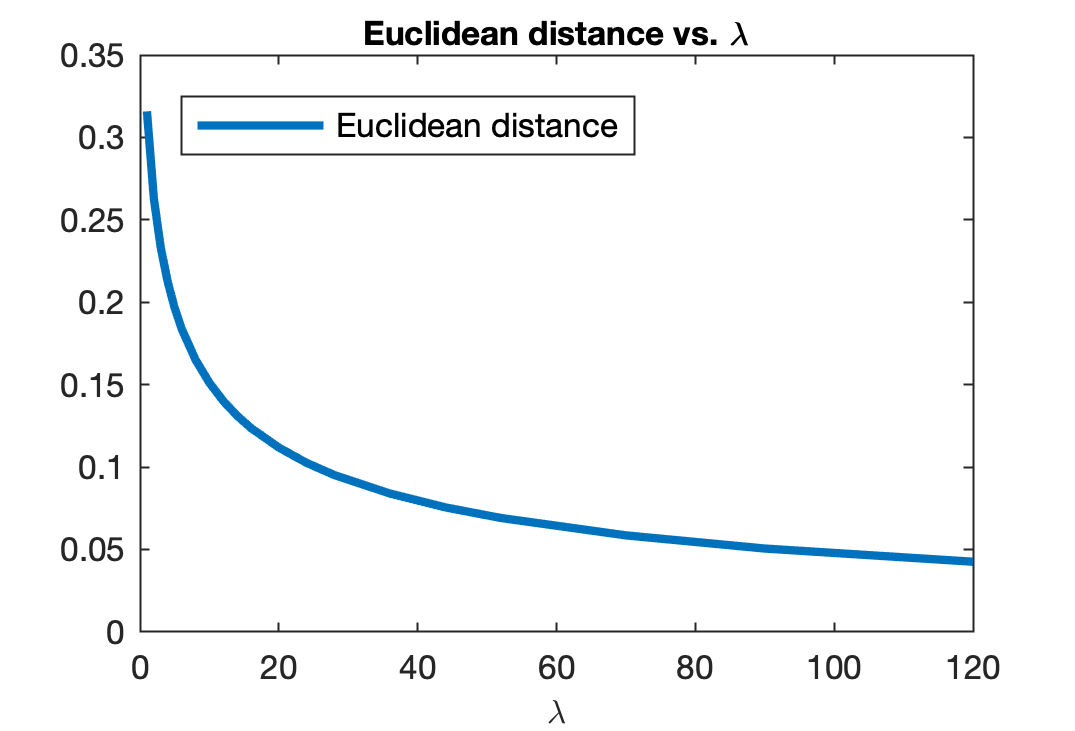}
  \captionof{figure}{Distance metric vs. $\lambda$}
  \label{fig:Distance_lambda}
\end{minipage}
\end{figure}

The Kullback--Leibler (K--L) divergence  introduced in \citet{kullback1951information},
is also known as relative entropy or information deviation. 
It measures the divergence of the distribution $\rho_{1}$ from
the target $\bar{\rho}_{1}$, the more similar the two distributions
are, the smaller the relative entropy will be. This measurement is
widely used in Machine Learning to compare two densities because it
has the following advantages: $1)$ this function is non-negative;
$2)$ for a fixed distribution $\bar{\rho}_{1}$, $C(\rho_{1},\bar{\rho}_{1})$
is convex in $\rho_{1}$; $3)$ $C(\rho_{1},\bar{\rho}_{1})=0$
if and only if $\rho_{1}=\bar{\rho}_{1}$ everywhere. There
are also caveats to the implementation of this penalty function. We
may face \textit{$0\log0$ }or \textit{division by zero }cases in
practice; to address this, we can replace zero with an infinitesimal positive
value.

In this case, the penalty functional is defined as 
\begin{equation}
C(\rho_{1},\bar{\rho}_{1})=\int_{\mathbb{R}} \lambda \rho_{1}(x)\ln\left(\frac{ \rho_{1}(x)}{\bar{\rho}_{1}(x)}\right)dx,
\label{KL}
\end{equation}
and the dual problem \eqref{eq:dual with phi_1} can be expressed explicitly as 
\begin{equation*}
V(\rho_{0},\bar{\rho}_{1}) =\sup_{\phi_{1}} \left\{ -\int_{\mathbb{R}} \lambda \exp\left(-\frac{\phi_1}{\lambda}-1\right)\bar{\rho}_1
-\phi_{0} \rho_{0} dx\right\} .
\end{equation*}
In Figure \ref{fig:KL_divergence}, we compare the empirical  terminal density $\rho_1$ and the target $\bar{\rho}_1$ when $C(\rho_{1},\bar{\rho}_{1})$ is defined by \eqref{KL} and $F(\tilde{A},\tilde{B})=(\tilde{A}-0.2)^2 + (\tilde{B}-0.2)^2$. The initial wealth $x_0 = 5$ and we set $\lambda = 0.1$ in Figure \ref{fig:KL_5.4_0.6_la01} and $\lambda = 10$ in Figure \ref{fig:KL_5.4_0.6_la10}.
\begin{figure}[h]
    \centering
    \subfloat[$\bar{\rho}_{1}=\mathcal{N}(5.4, 0.6), \lambda = 0.1$\label{fig:KL_5.4_0.6_la01}]{{\includegraphics[scale=0.7]{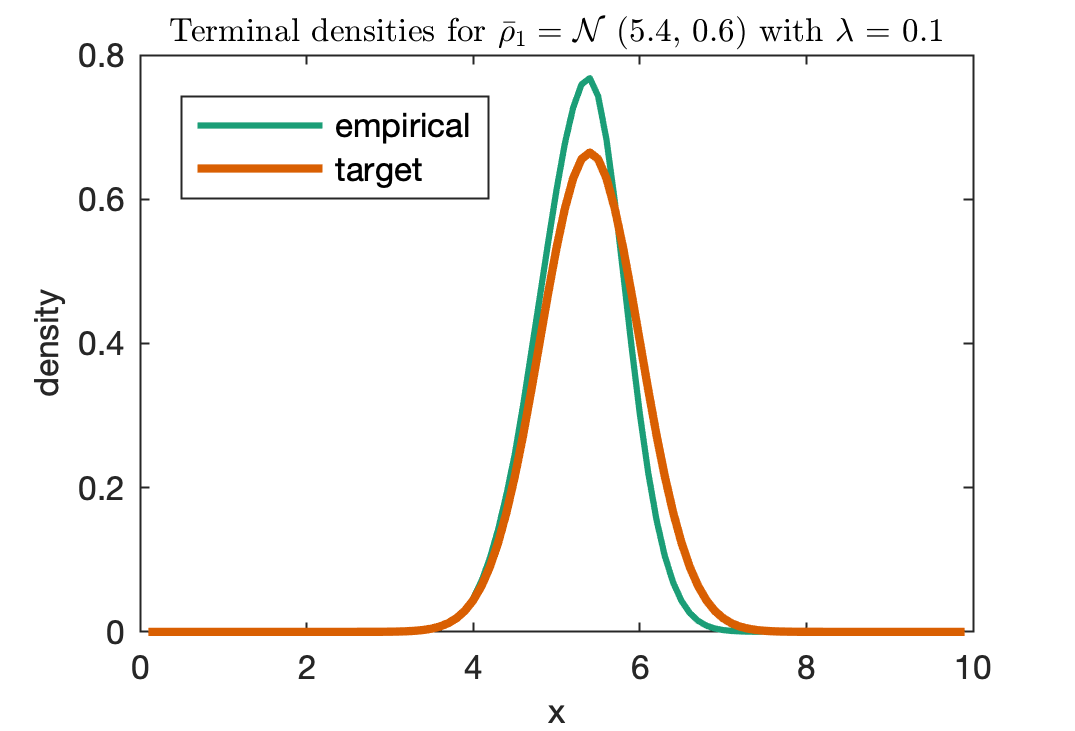} }}%
    \qquad \quad
    \subfloat[$\bar{\rho}_{1}=\mathcal{N}(5.4, 0.6), \lambda = 10$\label{fig:KL_5.4_0.6_la10}]{{\includegraphics[scale=0.7]{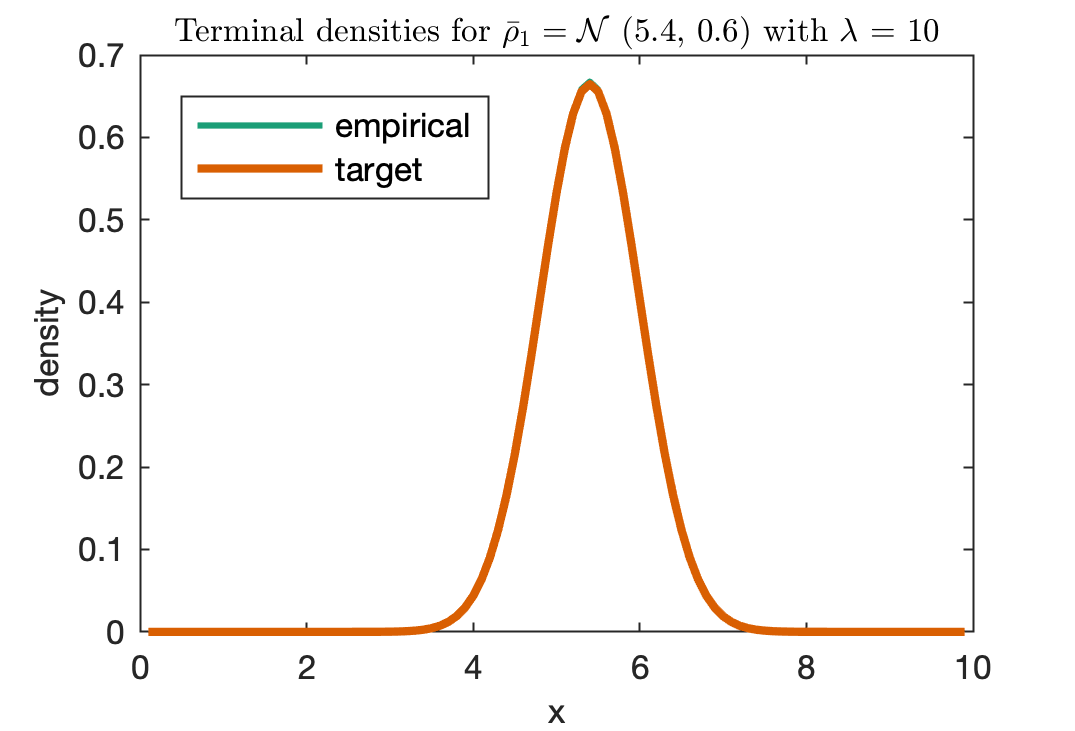} }}%
    \caption{K--L divergence as the penalty functional}%
    \label{fig:KL_divergence}%
\end{figure}

\subsection{Distribution of the wealth with cash saving \label{sec: consumption_part}}

In this section, we consider the cash saving during the investment process.
From previous parts, we have the constraint $(\tilde{B}^+)^{2}\leq\left\Vert \nu_t\right\Vert ^{2}\tilde{A}$.
However, when the prescribed target
$\bar{\rho}_{1}$ is not ambitious enough, we will find the optimal
drift $\tilde{B}^*$ is not saturated, i.e., $({\tilde{B}}^{*+})^2<\left\Vert \nu_t\right\Vert ^{2}\tilde{A}^*$ in \eqref{eq: optimal_A_B}.
In this case, we can actually attain a more ambitious distribution and have an accumulated cash saving $\int_{0}^{1}\left\Vert \nu_t\right\Vert \sqrt{\tilde{A}^*}-\tilde{B}^* dt$ during the investment process. 
Our goal in this section is to show that we can reach a better terminal distribution, in the sense that the terminal wealth has a higher expected value, when we take cash saving into account.

Denote $\left(C_{t}\right)_{t\in[0,1]}$ the accumulated cash saving
up to time $t$, and the evolution of $C_{t}$ is 
\begin{alignat*}{1}
dC_{t} & =\left(\left\Vert \nu_t\right\Vert \sqrt{\tilde{A}^*}-\tilde{B}^*\right)dt,\\
C_{0} & =0.
\end{alignat*}
If we add up the cash saving $C_{t}$ and the portfolio wealth $X_{t}$,
we can get a new process \textit{wealth with cash saving}.
Define $X^c_{t} \coloneqq X_{t}+C_{t}$, it is obvious to see that $X^c_{t}$ follows
the dynamics 
\begin{alignat*}{1}
dX^c_{t} & =\left\Vert \nu_t\right\Vert \sqrt{\tilde{A}^*}dt+\tilde{A}^{*^{\frac{1}{2}}} dW_{t},\\
X^c_{0} & =x_{0}.
\end{alignat*}
Denote by $p(t,x) \in \mathcal{P}$ the distribution of $X^c_{t}$ at time $t$, then
$p(t,x)$ satisfies the following Fokker-Planck equation 
\begin{alignat*}{1}
\partial_{t}p+\partial_{x}\left(\left\Vert \nu_t\right\Vert \sqrt{\tilde{A}^*}p\right)-\frac{1}{2}\partial_{xx}\left(\tilde{A}^*p\right) & =0,\\
p_0(x) & =\delta(x - x_{0}).
\end{alignat*}
Therefore, after solving for the optimal $\tilde{A}^{*}$, $\tilde{B}^{*}$
over time, we can find the densities of $X_{t}$ as well as $X^c_{t}$.
We keep using the squared Euclidean distance as the penalty functional
and $F(\tilde{A})=(\tilde{A}-0.2)^2 + (\tilde{B}-0.2)^2$ as the cost function.
Figure \ref{fig:With-consumption-allowed} compares the densities
for $X_{1}$ (terminal wealth), $X^c_{1}$ (terminal \textit{wealth with cash saving}) and the prescribed
target density. 
In Figure \ref{fig:consumption_5.1_0.5}, with a
rather conservative target $\bar{\rho}_{1}=\mathcal{N}(5.1,0.5)$,
although $\rho_{1}$ has attained the target, the distribution
for the \textit{wealth with cash saving} gathers at a higher value. When
we set a higher target $\bar{\rho}_{1} = \mathcal{N}(6,1)$, as in Figure \ref{fig:consumption_6_1},
we see there is no cash saved in the process since the paths for
$\rho_{1}$ and $p_{1}$ overlapped.
\begin{figure}[h]
    \centering
    \subfloat[$\bar{\rho}_{1}=\mathcal{N}(5.1,0.4)$\label{fig:consumption_5.1_0.5}]{{\includegraphics[scale=0.75]{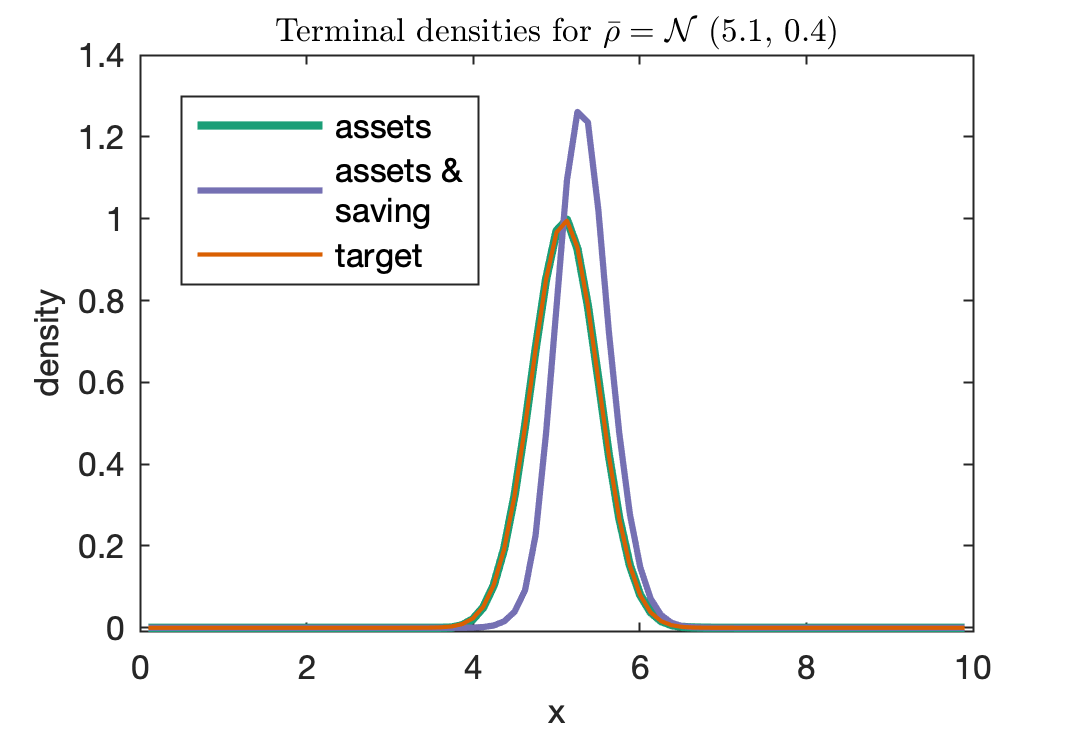} }}%
    \qquad \quad
    \subfloat[$\bar{\rho}_{1}=\mathcal{N}(6,1)$\label{fig:consumption_6_1}]{{\includegraphics[scale=0.75]{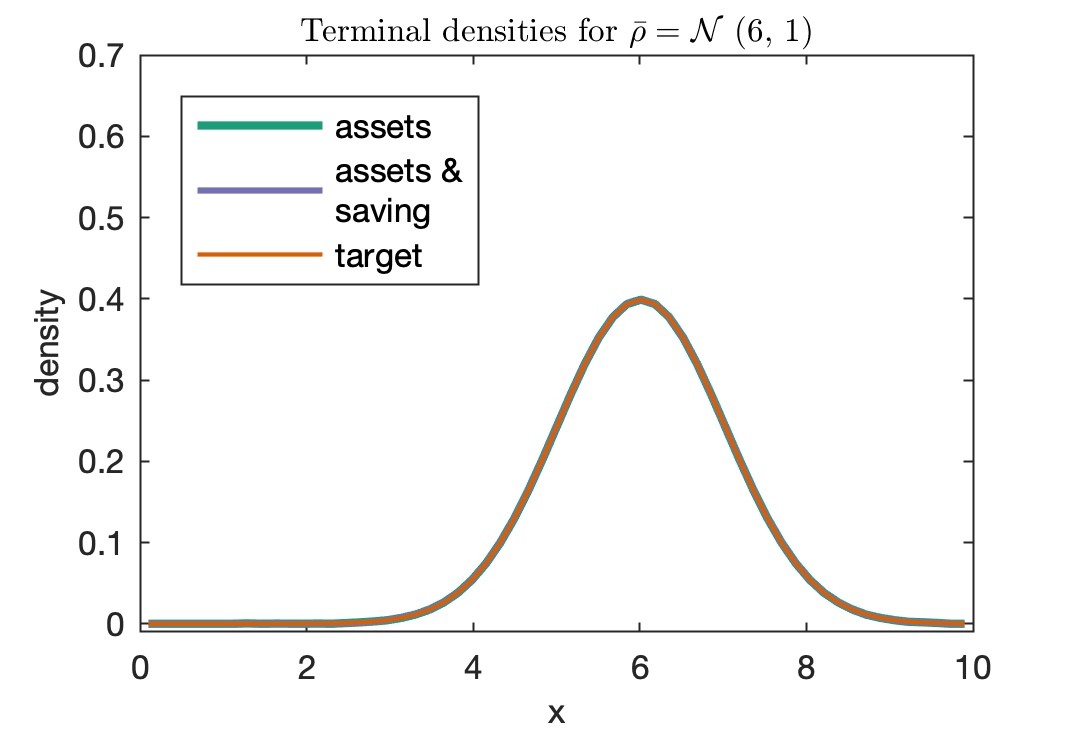} }}%
    \caption{Compare terminal distributions with or without cash saving}%
    \label{fig:With-consumption-allowed}%
\end{figure}

\subsection{Distribution of the wealth with cash input \label{sec: cash_input_part}}

As stated in Proposition \ref{prop:The-necessary-and-sufficient},
we always have $(\tilde{B}^+)^{2}\leq\left\Vert \nu_t\right\Vert ^{2}\tilde{A}$
for a self-financing portfolio. However, in this section, we remove the constraint $(\tilde{B}^+)^{2}\leq\left\Vert \nu_t\right\Vert ^{2}\tilde{A}$,
and we allow $\tilde{B}\in\mathbb{R}$ instead. 
Then the part $\left(\tilde{B}-\left\Vert \nu_t\right\Vert \sqrt{\tilde{A}}\right)^{+}$
can be interpreted as the extra cash we invest during the process. 
In this case, theoretically, we can attain any prescribed target distribution as we want
(see \citealt[Remark 2.3]{tan2013optimal}).
For the $\bar{\rho}_{1}$ which is unattainable by the self-financing portfolio, we can now attain
it with the help of cash input. However, to limit the use of cash, we design a cost function as follows,
\begin{alignat}{1}
F(\tilde{A},\tilde{B})=\begin{cases}
K(\tilde{B}^{2}-\left\Vert \nu_t\right\Vert ^{2}\tilde{A})+w\tilde{A}^{2}, & \quad\forall\tilde{B}>\left\Vert \nu_t\right\Vert \sqrt{\tilde{A}},\\
w\tilde{A}^{2}, & \quad\text{\ensuremath{\forall}}0\leq\tilde{B}\leq\left\Vert \nu_t\right\Vert \sqrt{\tilde{A}},\\
l\tilde{B}^{2}+w\tilde{A}^{2}, & \quad\forall\tilde{B}<0,
\end{cases}\label{eq:cost function penalizing}
\end{alignat}
where $K,w,l$ are positive constants. In the cost function (\ref{eq:cost function penalizing}),
we use the term $K(\tilde{B}^{2}-\left\Vert \nu_t\right\Vert ^{2}\tilde{A})$
to penalize the part $\left(\tilde{B}-\left\Vert \nu_t\right\Vert \sqrt{\tilde{A}}\right)^{+}$.
By varying $K$, we can control the strength of penalty and hence control the cash input flow. 
When $K$ is small, we are allowed to put in cash without being penalized excessively. 
When $K$ is large, we have to pay a high price for the cash input; consequently, the usage
is limited. The terms $w\tilde{A}^{2}$ and $l\tilde{B}^{2}$ add
coercivity to the function to ensure the existence of the solution,
we set $w,l$ to be small positive real values.

With the optimal drift $\tilde{B}^{*}\in \mathbb{R}$ and diffusion $\tilde{A}^{*} \in \mathbb{R}^+$, the dynamics of the wealth $X_t$ is
\begin{alignat*}{1}
dX_{t} & =\tilde{B}^*dt+\sqrt{\tilde{A}^*}dW_{t},\\
X_{0} & =x_{0}.
\end{alignat*}
If there is no cash input, the maximum drift is $\left\Vert \nu_t\right\Vert \sqrt{\tilde{A}}$.
Denote $\left(I_{t}\right)_{t\in[0,1]}$ the accumulated cash input
up to time $t$, and $I_{t}$ follows
the dynamics 
\begin{alignat*}{1}
dI_{t} & =\left(\tilde{B}^*-\left\Vert \nu_t\right\Vert \sqrt{\tilde{A}^*}\right)^{+}dt,\\
I_{0} & =0.
\end{alignat*}
Define $X^I_{t}\coloneqq X_{t}-I_{t}$ as the \textit{path without the cash
input}. Then the dynamics of $X^I_{t}$ is 
\begin{alignat*}{1}
dX^I_{t} & =\min\left(\tilde{B}^*,\left\Vert \nu_t\right\Vert \sqrt{\tilde{A}^*}\right)dt+\sqrt{\tilde{A}^*}dW_{t},\\
X^I_{0} & =x_{0}.
\end{alignat*}
Let the density of $X^I_{t}$ be $q(t,x) \in \mathcal{P}$, then $q(t,x)$ follows the
following Fokker-Planck equation 
\begin{alignat*}{1}
\partial_{t}q+\partial_{x}\left[\min\left(\tilde{B}^*,\left\Vert \nu_t\right\Vert \sqrt{\tilde{A}^*}\right)q\right]-\frac{1}{2}\partial_{xx}\left(\tilde{A}^*q\right) & =0,\\
q_0(x) & =\delta(x - x_{0}).
\end{alignat*}
Finally, we can see the effect of cash input by comparing $\rho_{1}(x)$ and $q_{1}(x)$.

\subsubsection{Attainable target}

In the first example, we aim at the terminal distribution $\bar{\rho}_{1}=\mathcal{N}(6,1)$, 
which is attainable by the self-financing portfolio.
We use the squared Euclidean distance as the penalty functional and
equation \eqref{eq:cost function penalizing} as the cost function.
Figure \ref{fig:cash input for 6_1} demonstrates the time evolution
of $q(t,x)$ (assets) and $\rho(t,x)$ (assets and cash input), and it compares
$q_{1}(x),\rho_{1}(x)$ and $\bar{\rho}_{1}(x)$ for various
$K$ values. 
At the beginning, we set the coefficient $K=0.5$ in Figure \ref{fig:cash input 6_1 K0.5}.
There is a clear difference between the paths for \textit{assets} and \textit{assets
and cash}, which means we have input a significant amount of cash over time. 
As we increase the value of $K$ in Figure \ref{fig:cash input 6_1 K4}, the difference
between $q(t,x)$ and $\rho(t,x)$ becomes less obvious. When
$K=6$, the paths with or without cash input coincide in Figure \ref{fig:cash input 6_1 K6}
because the high cost has prevented the cash input in this context.
Since the target $\mathcal{N}(6,1)$ is attainable, we can still reach
it even without cash input, as shown in the second plot of \ref{fig:cash input 6_1 K6}.

\begin{figure}
\subfloat[$K=0.5$\label{fig:cash input 6_1 K0.5}]{\includegraphics[scale=0.7]{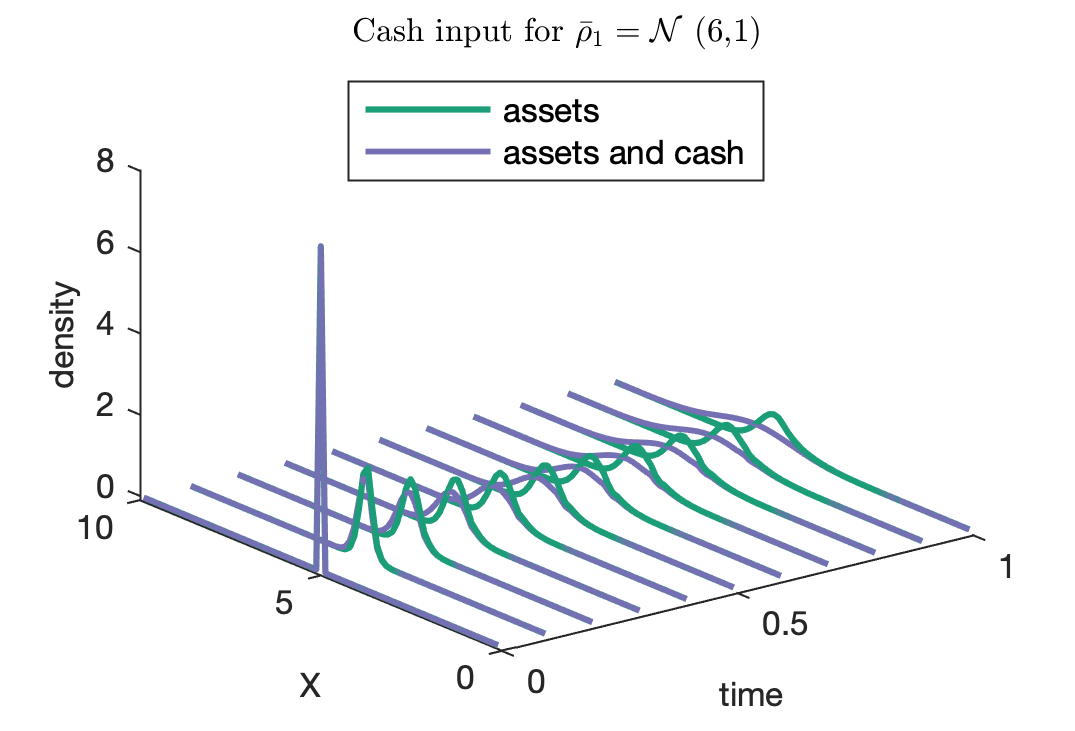} \includegraphics[scale=0.7]{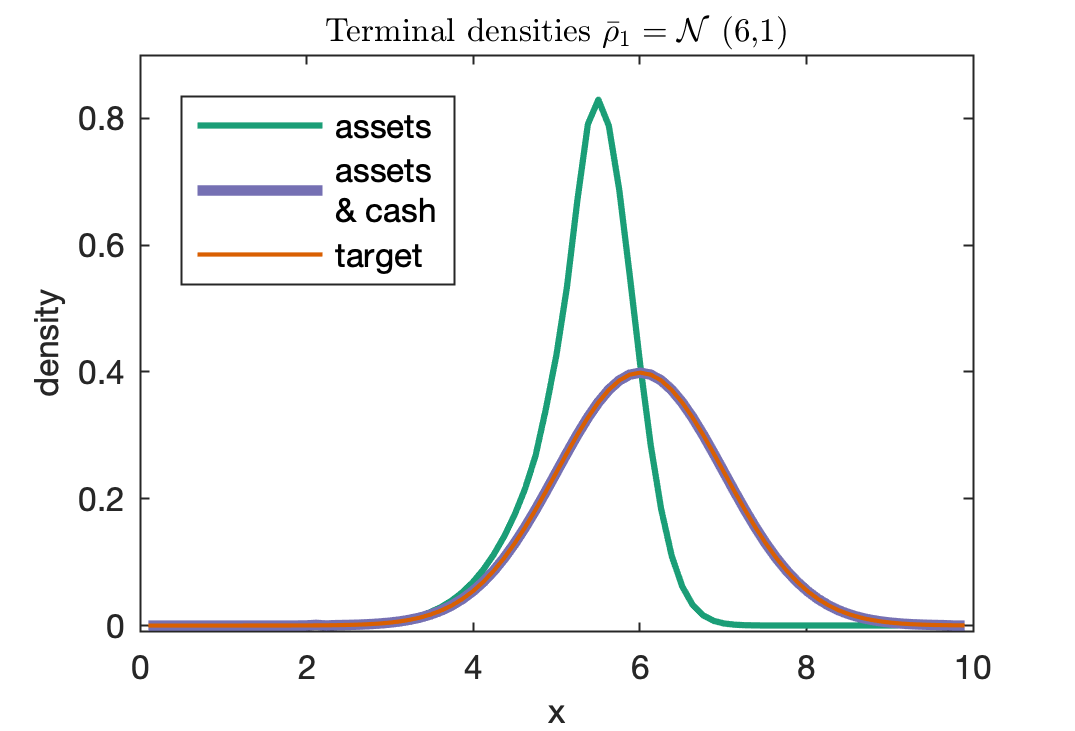}

}

\subfloat[$K=4$\label{fig:cash input 6_1 K4}]{\includegraphics[scale=0.7]{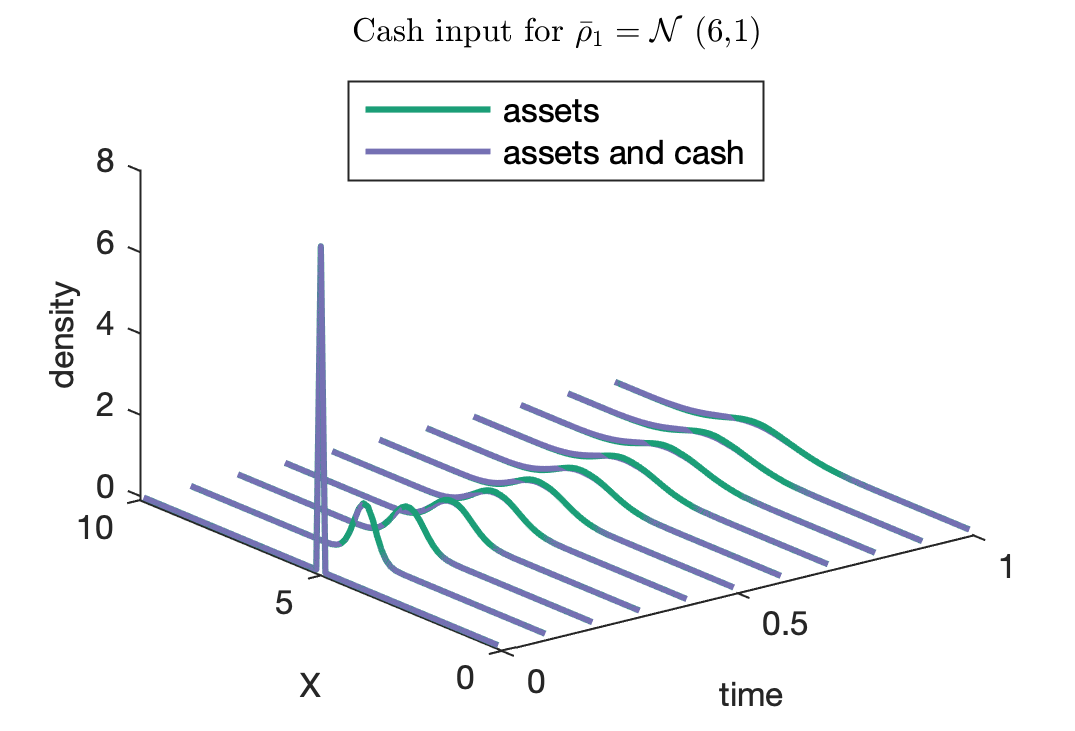} \includegraphics[scale=0.7]{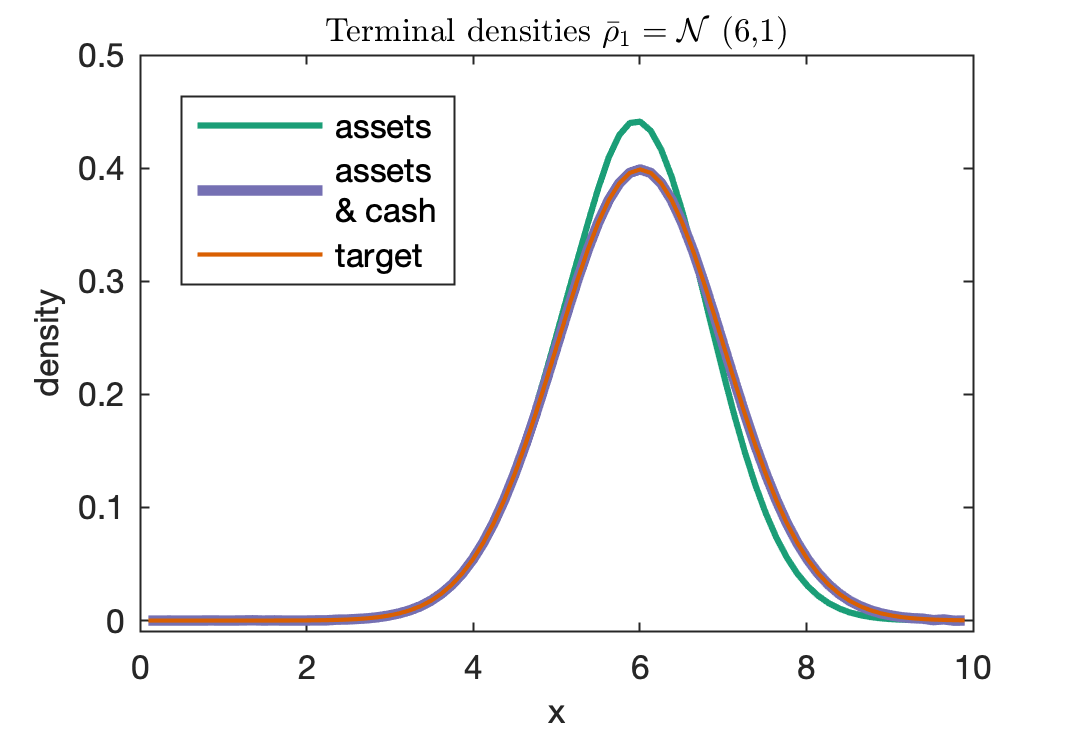}

}

\subfloat[$K=6$\label{fig:cash input 6_1 K6}]{\includegraphics[scale=0.7]{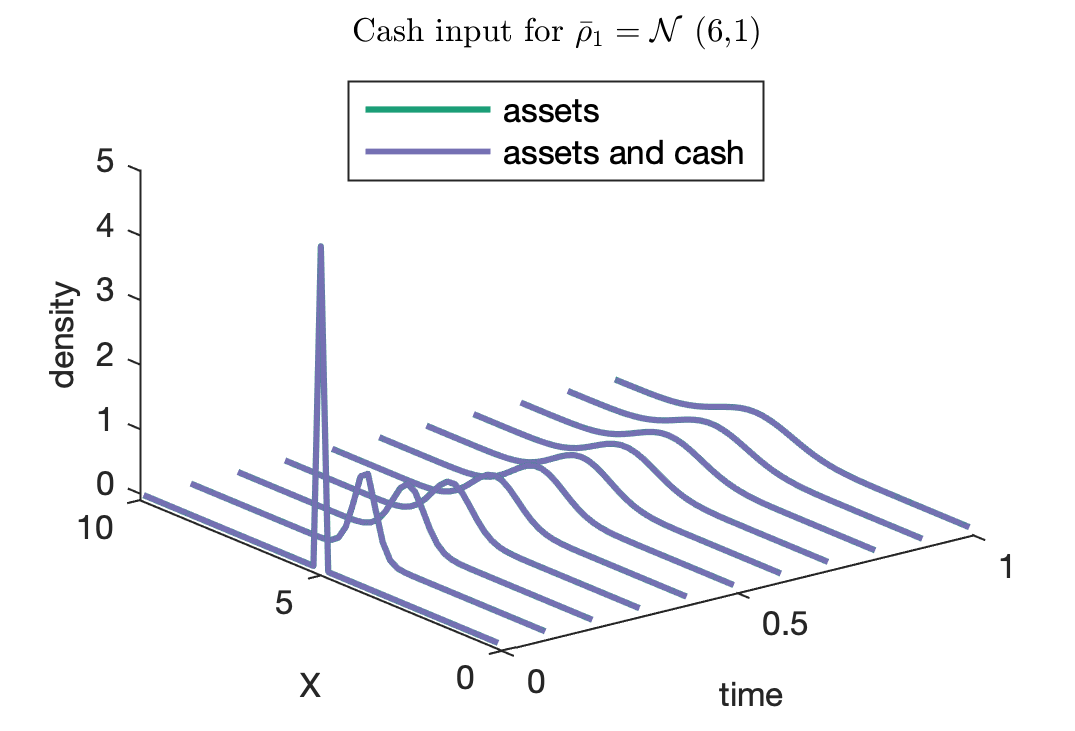} \includegraphics[scale=0.7]{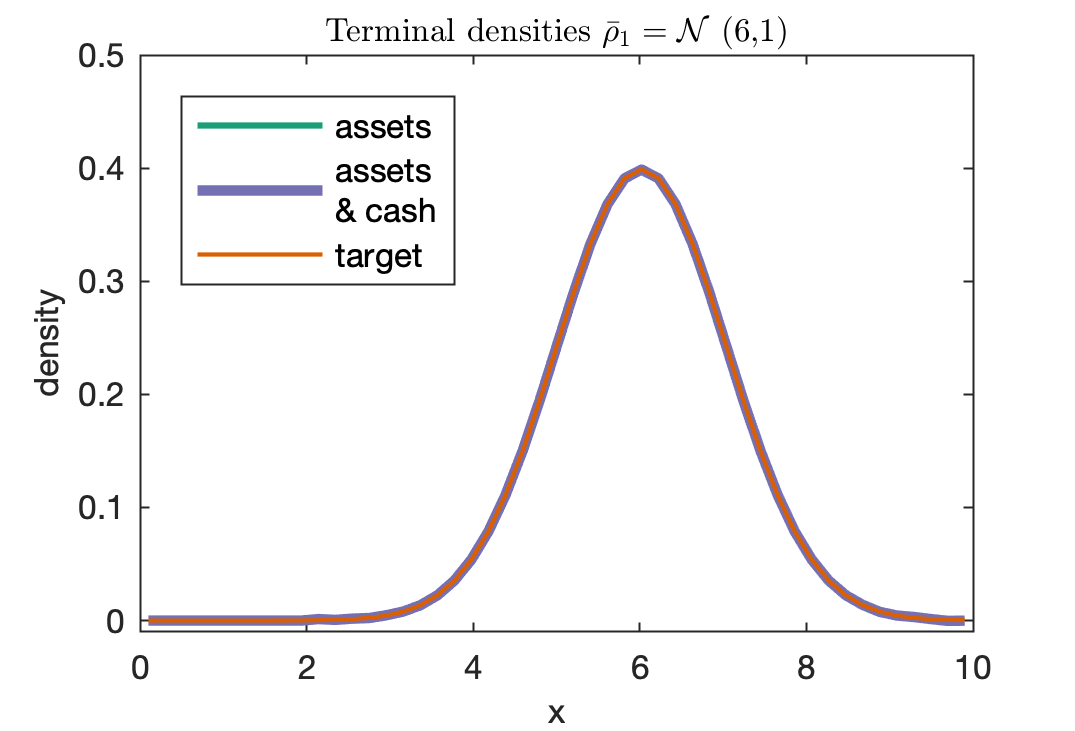}

}

\caption{Fixed $K$ for an attainable target: $\bar{\rho}_{1}=\mathcal{N}(6,1)$\label{fig:cash input for 6_1}}
\end{figure}

\subsubsection{Unattainable target}

To see the effect of cash input, here we demonstrate an example with
an unattainable target.  
For instance, we may target at a terminal distribution with
no left tail but a heavy right tail, in other words, there is very little risk
for the wealth to fall below some level. Therefore, we set $\bar{\rho}_{1}=Weibull\,(6,2)$
in Figure \ref{fig:K_t-0.8-Weibull}, where $P(x<4)$ is almost
zero. 
In this example, the coefficient $K$ in \eqref{eq:cost function penalizing} is not a constant anymore. 
Instead, we let $K(t):[0,1]\rightarrow\mathbb{R}^{+}$
be a function of time so that we can control the cash input flow over
time.
We define $K(t)=5$ for $t\in[0,0.8]$ and $K(t)=0.1$ for $t\in[0.8,1]$.
In the time-evolution plot (the left one of Figure \ref{fig:K_t-0.8-Weibull}), we can see that the paths for 
\textit{assets} and \textit{assets and cash} start to differentiate from $t=0.8$. 
Similarly, we can see the same effect in Figure \ref{fig:K_t 0.95}, where we set $\bar{\rho}_{1}=\mathcal{N}(6.5,1)$
and we define $K(t)=5$ for $t\in[0,0.95]$ and $K(t)=0.1$ for $t\in[0.95,1]$.
In these two examples, the targets $Weibull\,(6,2)$
and $\mathcal{N}(6.5,1)$ are unattainable under the
constraint $(\tilde{B}^+)^{2}\leq\left\Vert \nu_t\right\Vert ^{2}\tilde{A}$.
However, we can make the empirical terminal density $\rho_{1}$ reach $\bar{\rho}_{1}$ by inputting cash wisely.

\begin{figure}
\subfloat[{{{$K=5\:\forall t\in[0,0.8],K=0.1\:\forall t\in[0.8,1]$ for Weibull
$(6,2)$\label{fig:K_t-0.8-Weibull}}}}]{\includegraphics[scale=0.7]{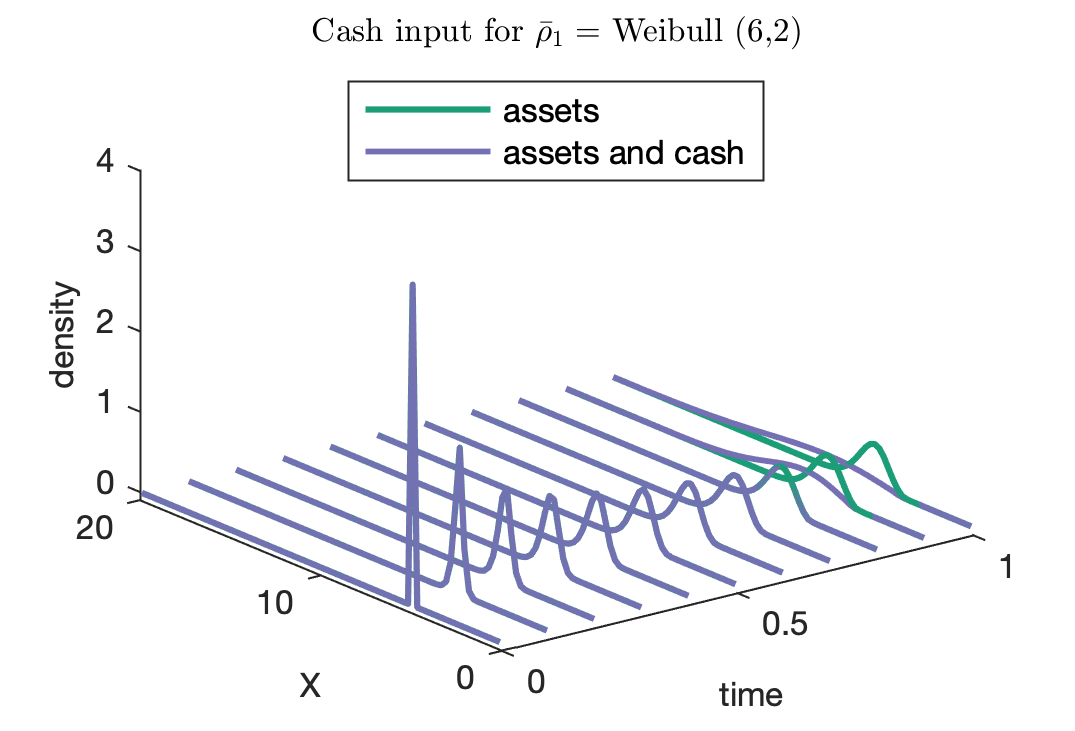}
\includegraphics[scale=0.7]{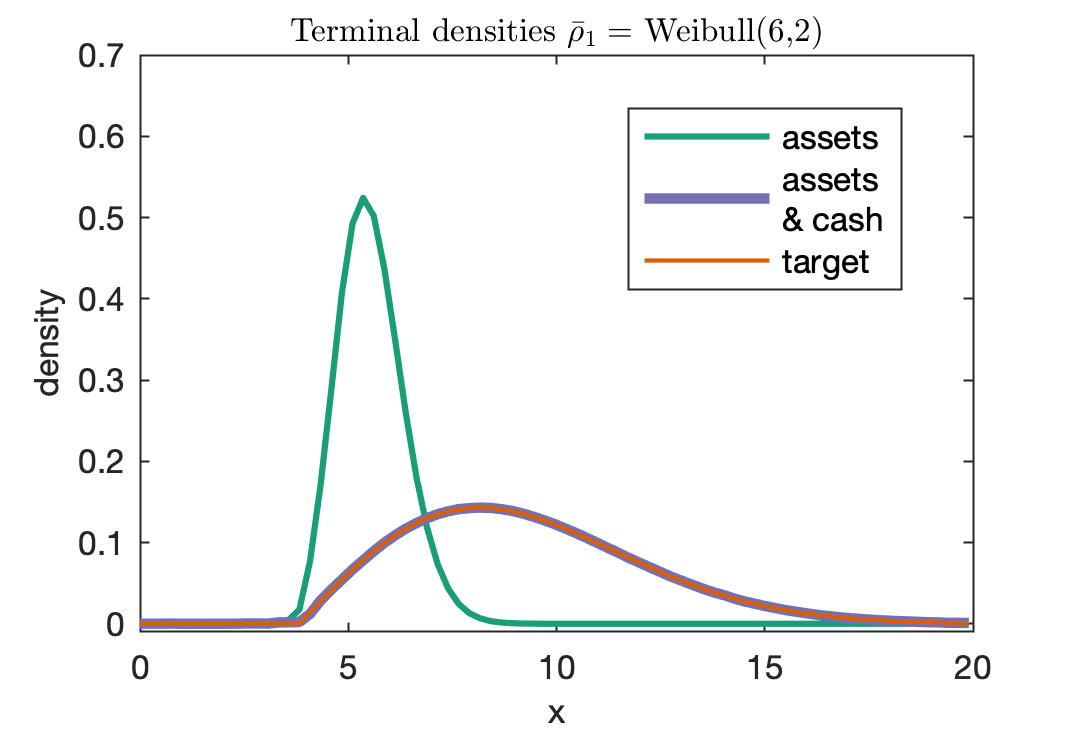}

}

\subfloat[{{{$K=5\:\forall t\in[0,0.95],K=0.1\:\forall t\in[0.95,1]$ for $\mathcal{N}(6.5,1)$\label{fig:K_t 0.95}}}}]{\includegraphics[scale=0.7]{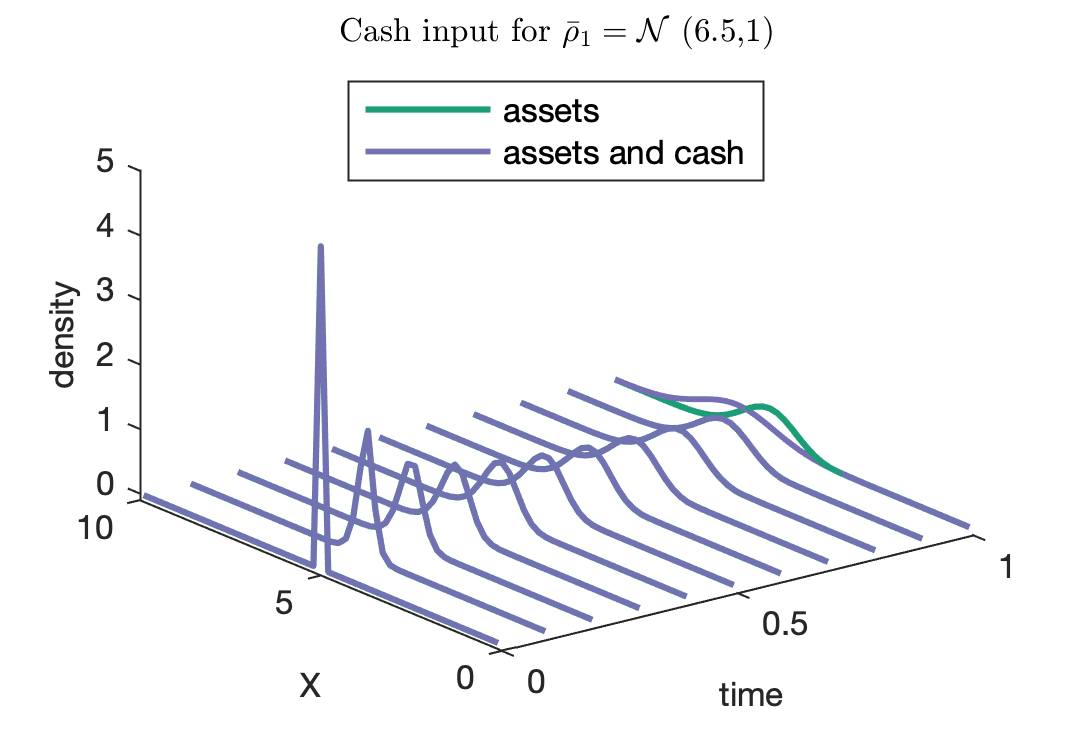}
\includegraphics[scale=0.7]{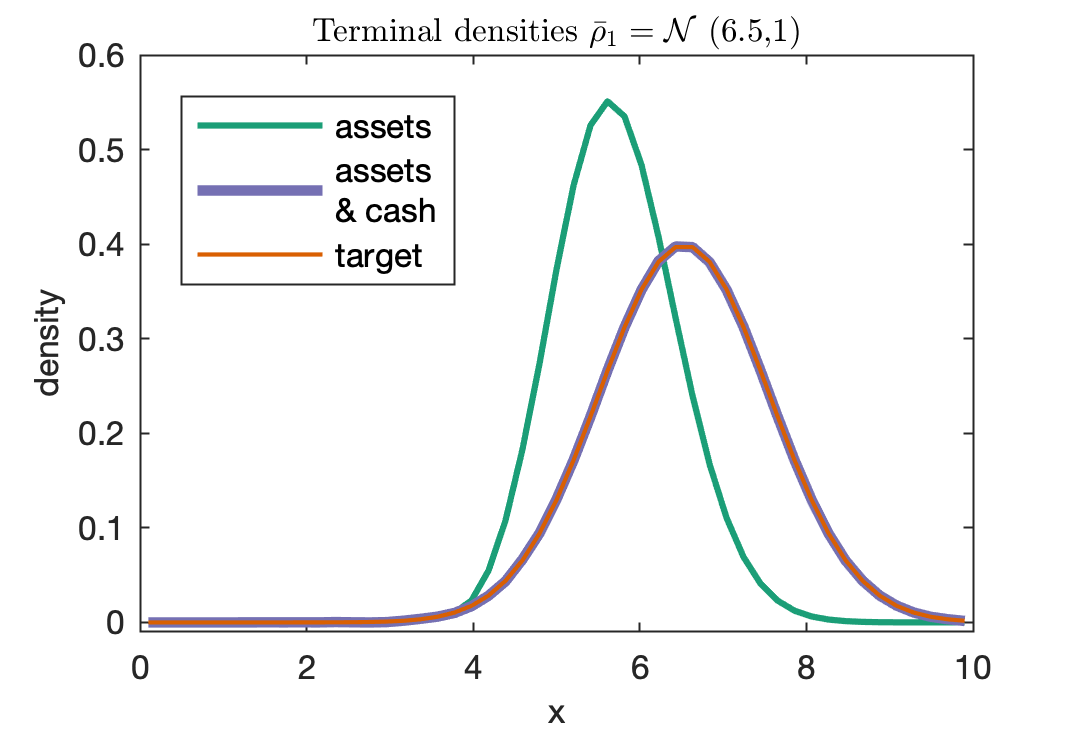}

}

\caption{$K(t)$ for unattainable targets}
\end{figure}

\section{Conclusion}
The ability to specify the whole distribution of final wealth of interest as portfolio optimization target gives a greater flexibility over classical objective functions such as expected utility or moment-based objectives such as the mean-variance framework and its extensions.
In this article, we construct a portfolio and the dynamics of the portfolio wealth is a semimartingale.
Starting from an initial wealth, by controlling the portfolio allocation process, 
we are able to steer the portfolio wealth to a prescribed distribution at the terminal time.
This problem is closely related to optimal mass transport (OMT). 
In the problem formulation, in addition to the conventional cost function in OMT, we also design a penalty functional to measure
the divergence of the empirical terminal density from the prescribed one. 
We take into consideration the possible consumption during the investment process, and show that we can actually 
reach a better terminal density when there is no consumption. 
When the target density is attainable, our problem can recover the classical OMT problem 
by choosing an indicator function as the penalty function. 
When the target terminal density is unattainable by the self-financing portfolio, 
we devise a strategy to reach it by allowing cash input during the investment process. 
We proved a duality result for the primal problem and solved it with a gradient descent
based algorithm. Our numerical results verify the accuracy and validity of this algorithm.

\bibliographystyle{chicago}
\bibliography{opt_pre_density}

\appendix
\section{Appendices}

\subsection{\label{sec:proof_AB_relation}}
\begin{proof}[Proof of Proposition \ref{prop:The-necessary-and-sufficient}] 
$\,$

Firstly, we prove the necessity. We can use eigen-decomposition and write
the covariance matrix $\Sigma_t=Q\Lambda Q^{^{\intercal}}=Q\Lambda^{\frac{1}{2}}\Lambda^{\frac{1}{2}}Q^{\intercal}$,
where $Q \in \mathbb{R}^{d \times d}$ and the $i$th column of $Q$ is
the eigenvector $q_{i}$ of $\Sigma_t$, and $\Lambda\in\mathbb{R}^{d\times d}$
is the diagonal matrix whose diagonal elements are the corresponding
eigenvalues, $\Lambda_{ii}=\lambda_{i}$. 

For any given $\alpha_{t} \in \mathbb{R}^d$, $A=\alpha_{t}^{\intercal}Q\Lambda^{\frac{1}{2}}\Lambda^{\frac{1}{2}}Q^{\intercal}\alpha_{t}x^{2}\rho$.
We define $\beta\coloneqq\alpha_{t}^{\intercal}Q\Lambda^{\frac{1}{2}}$,
then $A=\beta\beta^{\intercal}x^{2}\rho=\left\Vert \beta\right\Vert ^{2}x^{2}\rho$,
where $\left\Vert \cdot\right\Vert $ denotes the $L^{2}$ norm.
Similarly, $B=\alpha_{t}^{\intercal}Q\Lambda^{\frac{1}{2}}(Q\Lambda^{\frac{1}{2}})^{-1}\mu_t x\rho=\beta(Q\Lambda^{\frac{1}{2}})^{-1}\mu_t x\rho$.
Therefore we have the relationship between $A$ and $B$ as 
\[
B^{2}=\left(\beta(Q\Lambda^{\frac{1}{2}})^{-1}\mu_t\right)^{2}x^{2}\rho^{2} 
\leq \left\Vert \beta\right\Vert ^{2} \left\Vert (Q\Lambda^{\frac{1}{2}})^{-1}\mu_t \right\Vert ^{2}x^{2}\rho^{2}
=A\rho\left\Vert (Q\Lambda^{\frac{1}{2}})^{-1}\mu_t \right\Vert ^{2},
\]
\[
\left\Vert (Q\Lambda^{\frac{1}{2}})^{-1}\mu_t \right\Vert ^{2}A\geq\frac{B^{2}}{\rho}.
\]
Define $\nu_t \coloneqq(Q\Lambda^{\frac{1}{2}})^{-1}\mu_t = \Sigma_t^{-\frac{1}{2}}\mu_t$, we can write
the above inequality as $A\geq\frac{B^{2}}{\left\Vert \nu_t \right\Vert ^{2}\rho}$.

For given $(\rho, B, A)$ satisfying $A\geq\frac{B^{2}}{\left\Vert \nu_t \right\Vert ^{2}\rho}$,
we want to show that there exists $\alpha_{t}\in\mathbb{R}^{d}(d>1)$,
such that $A=\alpha_{t}^{\intercal}\Sigma_t \alpha_{t}x^{2}\rho$,
$B=\alpha_{t}^{\intercal}\mu_t x\rho$. 
First of all, since $\frac{A}{\rho} \geq 0$,
there exists a vector $\beta\in\mathbb{R}^{1\times d}$ whose
norm satisfies $\left\Vert \beta\right\Vert ^{2}=\frac{A}{x^{2}\rho}$.
Then $\frac{B^{2}}{\left\Vert \nu_t\right\Vert ^{2}\rho}\leq A$ will
be equivalent to $B^{2}\leq\left\Vert \beta\right\Vert ^{2}\left\Vert \nu_t \right\Vert ^{2}x^{2}\rho^{2}$.
With Cauchy--Schwarz inequality, 
$(\beta\nu_t)^{2}x^{2}\rho^{2}\leq\left\Vert \beta\right\Vert ^{2}\left\Vert \nu_t\right\Vert ^{2}x^{2}\rho^{2}$
holds. Therefore there exists a vector $\beta \in \mathbb{R}^{1\times d}(d>1)$ such that $B=\beta\nu_t x\rho$
and $\left\Vert \beta\right\Vert ^{2}=\frac{A}{x^{2}\rho}$. With
this $\beta$, there exists an $\alpha_{t}=Q\Lambda^{-\frac{1}{2}}\beta^{\intercal}$.

The case for dimension $d=1$ is trivial, hence omitted here.
\end{proof}

\subsection{\label{sec:extend_non_compact}}
\begin{prop}
We denote $K_{0}$ the set of $(u,b,a,r)\in C_{b}(\mathcal{E}, \mathbb{R}\times\mathbb{R}\times\mathbb{R}\times\mathbb{R})$ that  can be represented by $\phi \in C_b^{1,2}(\mathcal{E})$.
Then we have 
\[
\inf_{(\rho,B,A,\rho_1)\in C_{b}^{*}(\mathcal{E};\mathbb{R}\times\mathbb{R}\times\mathbb{R}\times\mathbb{R})}\left\{ \alpha^{*}(\rho,B,A,\rho_1)+\beta^{*}(\rho,B,A,\rho_1)\right\} 
=\inf_{(\rho,B,A,\rho_1)\in\mathcal{M}(\mathcal{E};\mathbb{R}\times\mathbb{R}\times\mathbb{R}\times\mathbb{R})}\left\{ \alpha^{*}(\rho,B,A,\rho_1)+\beta^{*}(\rho,B,A,\rho_1)\right\} .
\]
\end{prop}
\begin{proof}
Following closely the argument in \citet[Section 1.3]{villani2003topics},
we define $C_{0}(\mathcal{E})$ the space of continuous functions on $\mathcal{E}$, going to $0$
at infinity. For $(\rho,B,A,\rho_1)\in C_{b}^{*}(\mathcal{E};\mathbb{R}\times\mathbb{R}\times\mathbb{R}\times\mathbb{R})$, we decompose $(\rho,B,A,\rho_1)=(\hat{\rho},\hat{B}, \hat{A}, \hat{\rho}_1)+(\delta\rho,\delta B,\delta A, \delta \rho_1)$,
where $(\hat{\rho},\hat{B}, \hat{A}, \hat{\rho}_1) \in \mathcal{M}(\mathcal{E};\mathbb{R}\times\mathbb{R}\times\mathbb{R}\times\mathbb{R})$. For any $(u,b,a,r)\in C_{0}(\mathcal{E}; \mathbb{R} \times \mathbb{R}\times \mathbb{R}\times \mathbb{R})$,
we have $\langle (u,b,a,r),(\delta\rho,\delta B,\delta A, \delta \rho_1) \rangle = 0$.

Because $\mathcal{M}(\mathcal{E}; \mathbb{R})$ is a subset of $C_{b}^{*}(\mathcal{E}; \mathbb{R})$, we naturally have 
\begin{equation}
\inf_{(\rho,B,A,\rho_1)\in C_{b}^{*}(\mathcal{E};\mathbb{R}\times\mathbb{R}\times\mathbb{R}\times\mathbb{R})}\left\{ \alpha^{*}(\rho,B,A,\rho_1)+\beta^{*}(\rho,B,A,\rho_1)\right\} 
\leq\inf_{(\rho,B,A,\rho_1)\in\mathcal{M}(\mathcal{E};\mathbb{R}\times\mathbb{R}\times\mathbb{R}\times\mathbb{R})}\left\{ \alpha^{*}(\rho,B,A,\rho_1)+\beta^{*}(\rho,B,A,\rho_1)\right\} .\label{eq:inequality_1}
\end{equation}
\begin{comment}
Notice that the right hand side of the inequality (\ref{eq:inequality_1})
can be written as 
$\inf_{(\hat{\rho},\hat{B}, \hat{A}, \hat{\rho}_1)\in\mathcal{M}}\left\{ \alpha^{*}(\hat{\rho},\hat{B}, \hat{A}, \hat{\rho}_1)+\beta^{*}(\hat{\rho},\hat{B}, \hat{A}, \hat{\rho}_1)\right\} $.
\end{comment}
Now we look at the opposite direction of inequality (\ref{eq:inequality_1}).
For $\alpha^{*}$, we have 
\begin{align*}
\alpha^{*}(\rho, B, A, \rho_1) 
& =\sup_{(u,b,a,r)\in C_{b}(\mathcal{E}; \mathbb{R}\times\mathbb{R}\times\mathbb{R}\times\mathbb{R})}\Bigl\{ \int_{\mathcal{E}}ud\rho+bdB+adA+\Bigl[\int_{\mathbb{R}}rd\rho_{1}-C^{*}(r)\Bigr]: u+F^{*}(b,a)\leq0 \Bigr\} \\
 & \geq \sup_{(u,b,a,r)\in C_{0}(\mathcal{E}; \mathbb{R}\times\mathbb{R}\times\mathbb{R}\times\mathbb{R})}\Bigl\{ \int_{\mathcal{E}}ud\rho+bdB+adA+\Bigl[\int_{\mathbb{R}}rd\rho_{1}-C^{*}(r)\Bigr]: u+F^{*}(b,a)\leq0 \Bigr\} \\
 & = \sup_{(u,b,a,r)\in C_{0}(\mathcal{E}; \mathbb{R}\times\mathbb{R}\times\mathbb{R}\times\mathbb{R})}\Bigl\{ \int_{\mathcal{E}}
 ud\hat{\rho}+bd\hat{B}+ad\hat{A}+\Bigl[\int_{\mathbb{R}}rd \hat{\rho}_{1}-C^{*}(r)\Bigr]: u+F^{*}(b,a)\leq0 \Bigr\} \\
 & =\alpha^{*}( \hat{\rho},\hat{B},\hat{A},\hat{\rho}_1).
\end{align*}
We know $\beta^{*}(\hat{\rho},\hat{B},\hat{A},\hat{\rho}_1)=0$ if $(\hat{\rho},\hat{B},\hat{A},\hat{\rho}_1)$
satisfies (\ref{eq:Lagrangian penalty 1-1}), and $\beta^{*}(\hat{\rho},\hat{B},\hat{A},\hat{\rho}_1)=+\infty$
otherwise. When $\beta^*$ is finite, 
$$
\int_{\mathcal{E}}ud\hat{\rho}+bd\hat{B}+ad\hat{A}+\int_{\mathbb{R}}rd\hat{\rho}_{1}-\phi_{0}d\rho_{0} = 0 \qquad \forall (u, b, a, r)\in K_{0}.
$$
Then we have
\begin{align*}
\beta^{*}(\hat{\rho},\hat{B},\hat{A},\hat{\rho}_1) 
& = \sup_{(u, b, a, r)\in C_{0}\cap K_{0}}\int_{\mathcal{E}}ud\hat{\rho}+bd\hat{B}+ad\hat{A}+\int_{\mathbb{R}}rd\hat{\rho}_{1}-\phi_{0}d\rho_{0}\\
& = \sup_{(u, b, a, r)\in C_{0}\cap K_{0}}\int_{\mathcal{E}}ud\rho+bdB+adA+\int_{\mathbb{R}}rd\rho_{1}-\phi_{0}d\rho_{0}\\
& \leq\sup_{(u, b, a, r)\in K_{0}}\int_{\mathcal{E}}ud\rho+ bdB + adA +\int_{\mathbb{R}}rd\bar{\rho}_{1}-\phi_{0}d\rho_{0}\\
& =\beta^{*}(\rho, B, A, \rho_1).
\end{align*}
This completes the proof.
\end{proof}

\end{document}